\documentclass[a4paper,10pt]{article}      

\usepackage[latin1]{inputenc}
\usepackage{enumerate}

\usepackage{amsmath}
\usepackage{amsfonts,amssymb,MnSymbol}
\usepackage{amsthm}

\usepackage[all,cmtip]{xy}
\newdir{ >}{{}*!/-5pt/@{>}}
\usepackage{array}
\usepackage{color}
\usepackage{ulem}
\usepackage{dsfont}
\usepackage{bbm}

\usepackage{tikz-cd}

\usepackage[inline]{enumitem}

\newcounter{count}
\newtheorem{theorem}[count]{Theorem}
\newtheorem{lemma}[count]{Lemma}
\newtheorem{proposition}[count]{Proposition}

\theoremstyle{definition}
\newtheorem{definition}[count]{Definition}

\newtheorem{remark}[count]{Remark}
\newtheorem{notation}[count]{Notation}

\numberwithin{count}{section} 

\newcommand{\D}{\ensuremath{\mathcal{D}}}
\newcommand{\C}{\ensuremath{\mathcal{C}}}

\renewcommand{\O}{\ensuremath{\mathcal{O}}}
\renewcommand{\S}{\ensuremath{\mathcal{S}}}

\newcommand{\Sym}{\ensuremath{\mathsf{Sym}}}
\newcommand{\BSym}{\ensuremath{\mathsf{BSym}}}
\newcommand{\fgt}{\ensuremath{\mathsf{fgt}}}

\DeclareMathOperator{\colim}{colim}
\DeclareMathOperator{\Lan}{Lan}
\DeclareMathOperator{\Map}{Map}
\DeclareMathOperator{\Id}{Id}

\author{Lu\'is Alexandre Pereira}
\title{Cofibrancy of operadic constructions in positive symmetric spectra}

\begin{document}          \maketitle

\begin{abstract}
We show that when using the underlying positive model structure on symmetric spectra one obtains cofibrancy conditions for operadic constructions under much milder hypothesis than one would need for general categories. Our main result provides such an analysis for a key operation, the ``relative composition product'' $\circ_{\O}$ between right and left $\O$-modules over a spectral operad $\O$, and as a consequence we recover (and usually strengthen) previous results establishing the Quillen invariance of model structures on categories of algebras via weak equivalences of operads, compatibility of forgetful functors with cofibrations and Reedy cofibrancy of bar constructions.

Key to the results above are novel cofibrancy results for $n$-fold smash powers of positive cofibrant spectra (and the relative statement for maps). Roughly speaking, we show that such $n$-fold powers satisfy a (new) type of $\Sigma_n$-cofibrancy which can be viewed as ``lax $\Sigma_n$-free/projective cofibrancy'' in that it determines a larger class of cofibrations still satisfying key technical properties of ``true $\Sigma_n$-free/projective cofibrancy''.
\end{abstract}

\tableofcontents

\section{Introduction}

Operads provide a convenient way to codify many types of algebraic structures on a category $\C$, such as monoids, commutative monoids or, when $\C$ has extra structure, Lie algebras, $E_n$-algebras, among others. Indeed, any of these types of structures can be identified with the algebras in $\C$ over a specific operad. 

\vskip 5pt

When $\C$ is additionally a suitable model category it is then natural to ask whether the category of algebras over a fixed operad $\O$, denoted $\mathsf{Alg}_{\O}(\C)$, inherits a model structure from $\C$ and, moreover, just how compatible such a model structure on $\mathsf{Alg}_{\O}(\C)$ is with the underlying model structure on $\C$. Technical reasons then make it desirable for $\C$ and $\mathsf{Alg}_{\O}(\C)$ to be cofibrantly generated 
model categories (briefly, this means (trivial) cofibrations can be built via colimits from certain generating ones), and one quickly finds that the biggest obstacle to tackling the questions above is the fact that general colimits in $\mathsf{Alg}_{\O}(\C)$ are not underlying colimits in $\C$, so that proving properties of the (intended) cofibrations in $\mathsf{Alg}_{\O}(\C)$ requires a substantial amount of work.

More generally, related problems occur when studying other natural operadic constructions. Indeed, one of the most compact ways of describing operads is as the monoids over a certain monoidal structure $\circ$, the composition product, and many operadic constructions, such as right modules, left modules and algebras (which are special left modules ``concentrated in degree $0$''), are then derived from $\circ$. However, $\circ$ is an unusual monoidal structure which behaves quite differently with respect to each of its variables, in particular preserving colimits in the first variable but not in the second, and one then finds that studying operadic constructions in a model category context naturally requires answering the non obvious question of which cofibrations are actually preserved by $\circ$, and when.

\vskip 5pt

When dealing with a general model category $\C$ answering the questions above seems to require mild to severe cofibrancy conditions on the operad $\O$ itself (cf. \cite{WhYa15}). The main goal of this paper is to prove that for the category $\mathsf{Sp}^{\Sigma}$ of symmetric spectra, however, these questions can be answered while making minimal to no cofibrancy assumptions on $\O$, at least provided one uses the positive $S$ model structure as the underlying model structure on $\mathsf{Sp}^{\Sigma}$.

\subsection{Main results}

Positive model structures on spectra were first introduced by Mandell, May, Schwede and Shipley in \cite{MMSS} and soon after used by Shipley in \cite{Sh04} to establish the existence of a projective model structure of symmetric ring spectra where cofibrations are compatible with the forgetful functor from symmetric ring spectra to spectra. Since then, many other results have shown the usefulness of positive structures when studying algebras over an operad, such as the existence of projective model structures for algebras over any simplicial operad shown by Elmendorf and Mandell in \cite{ElmMan06}, strengthened to hold for any spectral operad by Harper in \cite{Ha09}, and the compatibility between cofibrations and the forgetful functor for more general operads shown by Harper and Hess in \cite{HaHe13}.

Our main result, Theorem \ref{CIRCO POS THM} below, follows this trend by establishing a quite thorough control of the way $\circ$ (or more generally, its relative version $\circ_{\O}$ for right and left $\O$-modules) interacts with cofibrations. We encourage the reader daunted by the technical nature of the result to first peruse Section \ref{CONS SEC}, where consequences of Theorem \ref{CIRCO POS THM} (including stronger versions of the results mentioned in the previous paragraph) are discussed.

\begin{theorem}\label{CIRCO POS THM}
Let $\O$ be an operad in $\mathsf{Sp}^{\Sigma}$ and consider the relative composition product
\[\mathsf{Mod}^r_{\O}\times \mathsf{Mod}^l_{\O}\xrightarrow{\minus \circ_{\O} \minus} \Sym.\]
Regard $\mathsf{Mod}^l_{\O}$ as equipped with the projective positive $S$ stable model structure and $\Sym$ as equipped with the $S$ stable model structure.

Suppose $f_2 \colon M \to \bar{M}$ is a cofibration between cofibrant objects in $\mathsf{Mod}^l_{\O}$. Then if the map $f_1 \colon N \to \bar{N}$ in $\mathsf{Mod}^r_{\O}$ is an underlying cofibration (resp. monomorphism) in $\Sym$, so is their pushout product with respect to $\circ_{\O}$,
\[
 M \circ_{\O} \bar{N} \bigvee_{M \circ_{\O} N} \bar{M} \circ_{\O} N \xrightarrow{f_1 \square^{\circ_{\O}} f_2 }\bar{M} \circ_{\O} \bar{N}
.\]
Further, $f_1 \square^{\circ_{\O}} f_2$ is also a weak equivalence if either $f_1$ or $f_2$ is.
\end{theorem}

Technically speaking, most of the ingredients needed for our proof of Theorem \ref{CIRCO POS THM} are adapted from arguments used in \cite{ElmMan06}, \cite{Ha09}, \cite{HaHe13} to prove the original versions of the results which we recover in Section \ref{CONS SEC}. However, two important new ingredients deserve special mention.

The first of these is found in Proposition \ref{FILT PROP}, which extends crucial filtrations of certain pushouts in $\mathsf{Alg}_{\O}(\C)$ used in \cite{ElmMan06}, \cite{Ha09}, \cite{HaHe13} by still providing such filtrations after composing with $M \circ_{\O} (\minus)$ for some $M \in \mathsf{Mod}^r_{\O}$. Note that as we do not assume $\C=\mathsf{Sp}^{\Sigma}$, these filtrations should be relevant in a general setting.

The second ingredient is a more thorough characterization of what makes positive model structures so convenient. It is well known that, for $A$ any $\Sigma_n$-spectrum and $X$ a positive $S$ cofibrant spectrum, there is a canonical weak equivalence
\begin{equation} \label{POS EQUIV EQ} 
(A \wedge X^{\wedge n})_{\Sigma_n} \sim (A \wedge X^{\wedge n})_{h \Sigma_n} 
\end{equation}
and, indeed, this key result essentially suffices to carry out the proofs in \cite{Sh04}, \cite{ElmMan06}. However, since (\ref{POS EQUIV EQ}) makes no explicit reference to cofibrations, one quickly finds it insufficient when trying to establish cofibrancy results in $\mathsf{Alg}_{\O}$. The natural way to fix this would be to guess that (\ref{POS EQUIV EQ}) ought to be a consequence of $X^{\wedge n}$ being built from free $\Sigma_n$-cells, or put in model category terminology, it being (genuinely) $\Sigma_n$-cofibrant (indeed, were that the case combining Remark \ref{PUSHPRODGEN REM} with \cite[Thm 5.3.7]{HSS} would yield (\ref{POS EQUIV EQ})). Unfortunately, this turns out to be false (cf. \cite{HaCo15}; also, check Remark \ref{TWOMON REM}), though fortunately not by much\footnote{In fact, such a result was ``proven'' in the author's thesis via an induction argument using incorrect base cofibrancy claims. A key impetus for this paper is to correct those base claims.}.
 Indeed, our second ingredient is a (new) type of ``lax $\Sigma_n$-cofibrancy'' in $(\mathsf{Sp}^{\Sigma})^{\Sigma_n}$, which we formally call $S$ $\Sigma$-inj $\Sigma_n$-proj cofibrancy, such that 
\begin{enumerate*}
\item[(i)] $X^{\wedge n}$ is ``lax $\Sigma_n$-cofibrant'' for positive cofibrant $X$; 
\item[(ii)] ``lax $\Sigma_n$-cofibrations'' share the key technical properties of (genuine) $\Sigma_n$-cofibrations.
\end{enumerate*}
The formal results follow.

\begin{theorem}\label{SIGMANPUSHPROD THM}
Let $\mathsf{Sp}^{\Sigma}$ be equipped with the {\it positive} $S$ stable model structure and 
$(\mathsf{Sp}^{\Sigma})^{\Sigma_n}$ with the $S$ $\Sigma$-inj $\Sigma_n$-proj stable model structure. 

Then for $f\colon A\to B$ a cofibration in $\mathsf{Sp}^{\Sigma}$ its $n$-fold pushout product \[f^{\square n} \colon Q_{n-1}^n(f) \to B^{\wedge n}\]
is a cofibration in $(\mathsf{Sp}^{\Sigma})^{\Sigma_n}$, which is a weak equivalence if $f$ is.

Furthermore, if $A$ is cofibrant in $\mathsf{Sp}^{\Sigma}$ then $Q_{n-1}^n(f)$ (resp. $f^{\wedge n} \colon A^{\wedge n} \to B^{\wedge n}$) is cofibrant (resp. cofibration between cofibrant objects) in $(\mathsf{Sp}^{\Sigma})^{\Sigma_n}$.
\end{theorem}

\begin{theorem}\label{BIQUILLEN THM}
Consider the bifunctor 
\[(\mathsf{Sp}^{\Sigma})^G\times (\mathsf{Sp}^{\Sigma})^G \xrightarrow{\minus \wedge_{G} \minus} \mathsf{Sp}^{\Sigma},\] 
where the first copy of $(\mathsf{Sp}^{\Sigma})^G$ is regarded as equipped with the $S$ $\Sigma$-inj $G$-proj stable model structure.
Then $\wedge_{G}$ is a left Quillen bifunctor if either:
\begin{enumerate}
\item[(a)] Both the second $(\mathsf{Sp}^{\Sigma})^G$ and the target $\mathsf{Sp}^{\Sigma}$ are equipped with the respective monomorphism stable model structures;
\item[(b)] Both the second $(\mathsf{Sp}^{\Sigma})^G$ and the target $\mathsf{Sp}^{\Sigma}$ are equipped with the respective $S$ stable model structures.
\end{enumerate}
\end{theorem}

In terms of the existent literature, Theorems \ref{SIGMANPUSHPROD THM} and \ref{BIQUILLEN THM}(a) are most closely related to \cite[Props. 4.28*, 4.29*]{HaCo15}, which they both significantly generalize and improve on from a technical standpoint (cf. Remark \ref{PUSHPRODADV REM}).
Further, Theorem \ref{BIQUILLEN THM} is strongly motivated by \cite[Thm 5.3.7]{HSS} (which, as hinted at above, implies the (genuine) $\Sigma_n$-cofibration analogue result).

\subsection{Consequences}\label{CONS SEC}

In this section we list a series of less technical results that can easily be deduced from Theorem \ref{CIRCO POS THM} (or, in the case of the first part of Theorem \ref{MODMODELEXIST THM}, its proof).

\begin{theorem}\label{MODMODELEXIST THM}
Let $\O$ be any operad in $\mathsf{Sp}^{\Sigma}$, and let $\mathsf{Sp}^{\Sigma}$, $\Sym$ be equipped with the respective positive $S$ stable model structures.

Then the respective \textbf{projective positive $S$ model structures} on $\mathsf{Alg}_{\O}$, $\mathsf{Mod}^l_{\O}$ exist and are simplicial model structures.

Further, if $\O\to \bar{\O}$ is a stable equivalence in each degree then the induce-forget adjunctions 
\[\bar{\O}\circ_{\O}(\minus) \colon \mathsf{Alg}_{\O} \rightleftarrows \mathsf{Alg}_{\bar{\O}}\colon \fgt,\qquad 
\bar{\O}\circ_{\O}(\minus) \colon \mathsf{Mod}^l_{\O} \rightleftarrows \mathsf{Mod}^l_{\bar{\O}}\colon \fgt \]
are Quillen equivalences.
\end{theorem}

In the case of algebras over the commutative operad, Theorem \ref{MODMODELEXIST THM} was first proven in \cite{Sh04}, and for general simplicial operads in \cite{ElmMan06}. A result nearly identical to Theorem \ref{MODMODELEXIST THM} was the main result of \cite{Ha09}. Our result is a slight generalization of the latter in the sense that our model structure on $\mathsf{Mod}^l_{\O}$ has a larger class of cofibrations (cf. the discussion preceding \cite[Thm. 1.3]{Ha09}).

\begin{theorem}\label{FORGETFUL THM}
Let $\O$ be an operad in $\mathsf{Sp}^{\Sigma}$ which is $S$ cofibrant in $\Sym$. Then, equipping $\mathsf{Alg}_{\O}, \mathsf{Mod}^l_{\O}$
with their respective projective positive $S$ stable model structures and 
$\mathsf{Sp}^{\Sigma}, \Sym$
with their respective $S$ stable model structures, the forgetful functors
\[\fgt\colon \mathsf{Alg}_{\O} \to \mathsf{Sp}^{\Sigma},\qquad
  \fgt\colon \mathsf{Mod}^l_{\O} \to \Sym\]
send cofibrations between cofibrant objects to cofibrations between cofibrant objects.
\end{theorem}

In the case of algebras over the commutative operad Theorem \ref{FORGETFUL THM} was first proven in \cite{Sh04}, and extended to algebras and modules over general operads satisfying some cofibrancy conditions in \cite{HaHe13}. Our result improves on the latter by relaxing the cofibrancy conditions on the operad.

\begin{theorem}\label{BARCONSTRUCTION THM}
Suppose $\Sym$ is equipped with the positive $S$ stable model structure and consider an operad $\O$ in $\mathsf{Sp}^\Sigma$, right $\O$-module $M$ and a left $\O$-module $N$ such that the unit map 
$\mathcal{I} \to \O$ (resp. $M$ and $N$) is an underlying cofibration (resp. are cofibrant objects) in $\Sym$. Then the bar construction
\[B_n(M,\O,N)=M \circ \O^{\circ n} \circ N\] 
is Reedy cofibrant with respect to the model structure on $\Sym$.
\end{theorem}

A very similar result to Theorem \ref{BARCONSTRUCTION THM} was first proven in \cite{HaHe13}. Our result improves it by using more general cofibrancy conditions and allowing $\O(0) \neq \**$.  

\begin{remark}
One advantage of Theorems \ref{MODMODELEXIST THM}, \ref{FORGETFUL THM} and \ref{BARCONSTRUCTION THM} versus the original results in \cite{Ha09}, \cite{HaHe13} that they generalize is that the cofibrancy conditions used are more \textit{consistent} across results. This makes it easier to use the results in tandem, a relevant feature in upcoming joint work between the author and Kuhn.
\end{remark}

\begin{theorem}\label{FIBERSEQ THM}
Suppose $A$ is projective positive $S$ cofibrant in $\mathsf{Alg}_{\O}$ or, more generally, in $\mathsf{Mod}^l_{\bar{\O}}$. Then the functor
\[\mathsf{Mod}^r_{\O}\xrightarrow{(\minus) \circ_{\O}A} \mathsf{Sp}^{\Sigma} 
\quad \text{or, more generally, } \quad 
\mathsf{Mod}^r_{\O}\xrightarrow{(\minus) \circ_{\O}A} \Sym\]
preserves homotopy fiber sequences.
\end{theorem}

\subsection{Directions for future work}

A key motivation for the work in this paper  comes from upcoming joint work between the author and Kuhn where we study certain filtrations built using 
$M \circ_{\O} (\minus)$ type functors. Since we need to iterate such functors while obtaining homotopically meaningful constructions (cf. Theorem \ref{FIBERSEQ THM}), it becomes necessary to understand how those functors interact with cofibrancy conditions. 
 
Additionally, there are two natural directions in which to try to generalize the results in this paper.

The first direction would be to extend Theorem \ref{CIRCO POS THM} to multicategories/colo\-red operads in $\mathsf{Sp}^{\Sigma}$, following \cite{ElmMan06}. In light of a recent preprint (\cite{WhYa15}) by White and Yau investigating when analogues of Theorems \ref{MODMODELEXIST THM}, \ref{FORGETFUL THM} hold for general categories, this seems likely to be a formal question.

A second direction would be to extend our main results to other categories. A natural candidate for such a generalization is given by the (simplicial) genuine $G$-symmetric spectra of Hausmann (\cite{Ha14}), as those share the underlying categories used in this paper. Such a generalization is the subject of upcoming joint work between the author and Hausmann.

\subsection{Outline of the paper}

Section \ref{BASICDEF} introduces the required basic notation and terminology.

Section \ref{SINJPROJMOD SEC} defines and proves the existence (Theorems \ref{SGSPEC THM}, \ref{GSTABLE THM}, \ref{SUSPEC THM}) of the three model structures on $(\mathsf{Sp}^{\Sigma})^{G}$ necessary to formulate Theorems \ref{SIGMANPUSHPROD THM} and \ref{BIQUILLEN THM}.

Section \ref{PROPERTIES SEC} proves the key properties of the $S$ $\Sigma$-inj $G$-proj stable model structures featured in Theorems \ref{SIGMANPUSHPROD THM} and \ref{BIQUILLEN THM}, namely those results themselves as well as two minor but essential ``change of group'' results (Propositions \ref{PROPSMASH} and \ref{PROPINDUCE}).

Section \ref{APPLICATIONS SEC} deals with proving Theorem \ref{CIRCO POS THM} and its ``corollaries'' Theorems \ref{MODMODELEXIST THM}, \ref{FORGETFUL THM}, \ref{BARCONSTRUCTION THM} and \ref{FIBERSEQ THM}. Key to this is subsection \ref{FILTRATIONS SEC} and in particular Proposition \ref{FILT PROP}, which improves crucial filtration results used in \cite{ElmMan06}, \cite{Ha09}, \cite{HaHe13}, among others.

\vskip 5pt

Much of the paper, namely Sections \ref{SINJPROJMOD SEC} and \ref{PROPERTIES SEC}, is devoted to building the notion of ``lax $\Sigma_n$-cofibrancy'' needed to state Theorems \ref{SIGMANPUSHPROD THM} and \ref{BIQUILLEN THM} and then proving those results. However, the reader interested only in Theorem \ref{CIRCO POS THM} (or its consequences) should be able to skip ahead to Section \ref{APPLICATIONS SEC}, provided he is willing to accept Theorems \ref{SIGMANPUSHPROD THM} and \ref{BIQUILLEN THM} (and Propositions \ref{PROPSMASH} and \ref{PROPINDUCE}) as given.

\vskip 10pt

\noindent{\textbf{Acknowledgments}}

\vskip 3pt

The author would like to thank Nick Kuhn for his encouragement and advice as well as John Harper, Mark Behrens and David White for many useful conversations.

\section{Basic definitions and notation}\label{BASICDEF}

We assume the reader is familiar with the basics on symmetric spectra (as found in \cite{HSS} or \cite{Sch97}) and cover in Sections \ref{GSPACES SEC}, \ref{GSPECTRA SEC} and \ref{STABLEPOSSTABLE SEC} only the minimum needed to establish notation and some less standard basic results.

Likewise, we assume the reader is familiar with the basics on cofibrantly generated model categories (as in \cite{Hov98}) and recall in Sections \ref{INJPROJ SEC} and \ref{PUSHPROD SECSEC} only two notions that play a key role for us: injective/projective model structures and left Quillen bifunctors.

\subsection{Pointed $G$-simplicial sets}\label{GSPACES SEC}

Throughout we let $(\mathsf{S}_{\**},\wedge,S^0)$ denote the closed monoidal category of pointed simplicial sets together with its monoidal structure $\wedge$ and unit $S^0$.

We will make use of the following standard notation:
\begin{itemize}
\item for $A$ a set, $A \cdot (\minus)$ denotes the (constant) coproduct over $A$ (cf. \cite{McL});
\item $\Delta^k$, $\partial \Delta^k$ and $\Lambda^k_l$ denote the standard, boundary and horn (unpointed) simplicial sets (cf. \cite[I.1]{GJ});
\item $X_+$ denotes the pointed simplicial set obtained by adding a disjoint base point to the (unpointed) simplicial set $X$;
\item $S^n=(\Delta^1/\partial \Delta^1)^{\wedge n}$ denotes the pointed $n$-sphere.
\end{itemize}

\begin{definition}
Let $G$ be a finite group. The category $\mathsf{S}_{\**}^G$ of \textit{pointed $G$-simplicial sets} is the category of functors $G\to \mathsf{S}_{\**}$.
\end{definition}

Given $X,Y$ in $\mathsf{S}_{\**}^G$, $X\wedge Y$ has a diagonal $G$-action and, giving $S^0$ the trivial action, $\wedge$ becomes a monoidal structure in $\mathsf{S}_{\**}^G$. In fact, one has the following.

\begin{proposition}\label{GCLOSEDMON PROP}
$(\mathsf{S}_{\**}^G,\wedge,S^0)$ form a closed symmetric monoidal category.

Further, both the left and right adjoint in the trivial-fixed point adjunction
\[\mathsf{triv} \colon \mathsf{S}_{\**} \rightleftarrows  \mathsf{S}_{\**}^G \colon (\minus)^G\]
are monoidal functors.
\end{proposition}

The less obvious half of Proposition \ref{GCLOSEDMON PROP} follows by noting that 
$X \wedge Y = \colim (\** \leftarrow X \vee Y \to X \times Y)$ and using the following (which we will need later).

\begin{proposition}\label{FIXEDPUSH PROP}
Any pushout diagram in $\mathsf{S}_{\**}^G$
\begin{equation}\label{PUSH DIAGRD}
\xymatrix@1{A \ar[r] \ar[d]_f& X \ar[d] \\ B \ar[r] & Y}
\end{equation}
with $f$ a monomorphism remains a pushout diagram after applying $(\minus)^G$.
\end{proposition}

\begin{proof}
Monomorphisms are transfinite compositions of maps adding a single orbit, so that one reduces to $f=G/H \cdot (\partial \Delta^k_+ \to \Delta^k_+)$. The claim is now clear.
\end{proof}

\begin{remark}\label{SPAC SIMP}
By the theory of enriched categories (see for example \cite[Chap. 3]{Ri14}), Proposition \ref{GCLOSEDMON PROP} implies that $\mathsf{S}_{\**}^G$ is a tensored and cotensored $\mathsf{S}_{\**}$-enriched category, and hence also simplicially enriched, tensored and cotensored. Explicitly, we note that the mapping space for $X,Y \in \mathsf{S}_{\**}^G$ is $\Map(X,Y)^G$, the $G$-fixed points of the conjugation action on the mapping space of the underlying $X,Y \in \mathsf{S}_{\**}$.
\end{remark}

\subsection{$G$-spectra}\label{GSPECTRA SEC}

Throughout $\Sigma$ will denote the usual skeleton of the category of finite sets and bijections. Explicitly, the objects of $\Sigma$ are the sets $\underline{m}=\{1,2,\cdots,m\}$ for $m\geq 0$.

\begin{definition}\label{SYMSEQ DEF1}
The category of \textit{symmetric sequences} in pointed simplicial sets is the category $\mathsf{S}_{\**}^\Sigma$ of functors from $\Sigma$ to $\mathsf{S}_{\**}$.
\end{definition}

\begin{remark}\label{UNPACK REM}
Unpacking Definition \ref{SYMSEQ DEF1}, a symmetric sequence $X$ consists of a sequence $X_m, m \geq 0$ of pointed simplicial sets, each with a left $\Sigma_m$-action.
 One then has inclusions $\mathsf{S}_{\**}^{\Sigma_m} \hookrightarrow \mathsf{S}_{\**}^{\Sigma}$,  which we often omit to simplify notation.
\end{remark}

\begin{definition}
The \textit{tensor product} $X \otimes Y$ of $X,Y \in \mathsf{S}_{\**}^{\Sigma}$ is defined by
\[(X\otimes Y)_m = \bigvee_{i+j=m}\Sigma_{m} \underset{\Sigma_i \times \Sigma_j}{\cdot} X_i \wedge Y_j\]
together with the obvious $\Sigma_m$-actions.
\end{definition}

The following is proven in \cite[Sec. 2.2]{HSS}.

\begin{proposition} $(\mathsf{S}_{\**}^{\Sigma},\otimes, \mathbbm{1})$ form a symmetric monoidal category where the unit $\mathbbm{1}$ is the sequence
$\mathbbm{1}$ such that $\mathbbm{1}_0=S^0$ and $\mathbbm{1}_m= \**$ for $m> 0$.
\end{proposition}

It is well known that the symmetric sequence $S$, the \textit{sphere spectrum}, defined by $S_m=S^m$ is a symmetric monoid with respect to $\otimes$. General theory then implies that modules over $S$ themselves form a symmetric monoidal category.

\begin{definition}
The category $\mathsf{Sp}^{\Sigma}$ of \textsl{symmetric spectra} is the category of modules over $S$ in $\mathsf{S}_{\**}^{\Sigma}$.
The {\it smash product} $X\wedge Y$ of $X,Y \in \mathsf{Sp}^{\Sigma}$ is the coequalizer
\[X\otimes S \otimes Y \rightrightarrows X \otimes Y \rightarrow X \wedge Y.\]
\end{definition}

\begin{remark}
Throughout we will need to consider spectra $X$ such that each level $X_m$ is acted on by multiple symmetric groups (e.g. when $X=Y^{\wedge n}$). To avoid confusion, we will reserve the letter $m$ for the structure index of spectra.
\end{remark}

\begin{definition}
Let $G$ be a finite group. The category $(\mathsf{Sp}^\Sigma)^G$ of \textit{$G$-spectra} is the category of functors $G\to \mathsf{Sp}^\Sigma$.
\end{definition}

Just as for pointed simplicial sets, the smash product $X\wedge Y$ of $X,Y\in (\mathsf{Sp}^\Sigma)^G$ has a diagonal $G$-action. The following is immediate.

\begin{proposition}\label{SPECTRAADJUNCTION PROP}
Both $(\mathsf{Sp}^\Sigma,\wedge,S)$ and $((\mathsf{Sp}^{\Sigma})^G,\wedge,S)$ form closed symmetric monoidal categories.
Further, all functors in the following adjunctions are monoidal
\[S \otimes (\minus) \colon \mathsf{S}_{\**} \rightleftarrows  \mathsf{Sp}^{\Sigma} \colon (\minus)_0, \phantom{i}
  S \otimes (\minus) \colon \mathsf{S}_{\**}^G \rightleftarrows  (\mathsf{Sp}^{\Sigma})^G \colon (\minus)_0, \phantom{i}
 \mathsf{triv}\colon \mathsf{Sp}^\Sigma \rightleftarrows  (\mathsf{Sp}^{\Sigma})^G \colon (\minus)^G.\]
\end{proposition}

\begin{remark}\label{SPEC SIMP}
The theory of enriched categories (\cite[Chap. 3]{Ri14}) implies that $(\mathsf{Sp}^\Sigma)^G$ is enriched, tensored and cotensored over both $\mathsf{S}_{\**}$ and $\mathsf{Sp}^{\Sigma}$, and hence also simplicially enriched, tensored and cotensored.

\end{remark}

\subsection{$S$ stable and positive $S$ stable model structures on $\mathsf{Sp}^{\Sigma}$} \label{STABLEPOSSTABLE SEC}

\begin{definition}\label{SSPEC DEF}
The {\it $S$ stable model structure} (resp. {\it positive $S$ stable model structure}) on $\mathsf{Sp}^{\Sigma}$
is the cofibrantly generated model structure such that 
\begin{itemize}
\item the generating cofibrations are the maps 
\[ S \otimes \left( \Sigma_m/H \cdot (\partial \Delta^k_+ \to \Delta^k_+)\right)\]
for $m\geq 0$ and any $H \leq \Sigma_m$ (resp. $m\geq 1$ and any $H \leq \Sigma_m$);
\item weak equivalences are the stable equivalences of spectra.
\end{itemize}
\end{definition}

\begin{remark}
Our terminology follows \cite{HSS}, \cite{Sh04} rather than \cite{Sch97}, \cite{Ha09} which refer to $S$ cofibrations as ``flat cofibrations''. 
However, we make no explicit use of the maps referred to in \cite{HSS}, \cite{Sh04}, \cite{Sch97}, \cite{Ha09} simply as ``cofibrations'', even though our results also apply to those given that they are a subclass of $S$ cofibrations.
\end{remark}

\begin{remark}
The proof of Proposition \ref{LATCHMON PROP} shows that a $S$ cofibration $A \to B$ is also a positive $S$ cofibration iff $A_0 \xrightarrow{\simeq} B_0$ is an isomorphism. We hence use positivity as a hypothesis in our results only if necessary and never as a conclusion, leaving it to the curious reader to check by direct calculation if positivity conclusions hold.
\end{remark}

\subsection{Injective and projective model structures}\label{INJPROJ SEC}

\begin{definition}
Let $\C$ be a model category, $M$ a monad on $\C$ and $\mathsf{Alg}_{M}$ the category of algebras over $M$.

The {\it injective model structure} on $\mathsf{Alg}_{M}$, if it exists, has as cofibrations (resp. weak equivalences) the underlying cofibrations (resp. weak equivalences) in $\C$.

The {\it projective model structure} on $\mathsf{Alg}_{M}$, if it exists, has as fibrations (resp. weak equivalences) the underlying fibrations (resp. weak equivalences) in $\C$.
\end{definition}

\begin{remark}\label{NATOC REM}
Since most usual model structures are cofibrantly generated, it is often easier to build projective structures (cf. \cite[Lemma 2.3]{SS00}) than injective ones. 
\end{remark}


\begin{remark}\label{TWOMON REM}
When in the presence of two monads, building injective structures does not in general commute with building projective structures.

A key example is given by comparing what we call the 
$\Sigma_m$-inj $G$-proj model structure on 
$\mathsf{S}_{\**}^{G\times \Sigma_m}$, 
built as the $\Sigma_m$-injective structure over the $G$-projective structure over the standard model structure on $\mathsf{S}_{\**}$, with the 
$G$-proj $\Sigma_m$-inj model structure on
$\mathsf{S}_{\**}^{G\times \Sigma_m}$, which reverses the two constructions.

The former is shown to exist in Proposition \ref{NGSPAC PROP}, and examining the generating cofibrations, listed when proving Proposition \ref{GSPAC PROP}, yields that cofibrations are those monomorphisms $A\hookrightarrow B$ adding only $G$-free simplices. Conversely, the latter is built by replacing the condition $H\cap G\times \**=\**$ in Proposition \ref{GSPAC PROP} with $H\subset \** \times \Sigma_m$, so that cofibrations are those monomorphisms adding only simplices with such $H$ as isotropies. 

The need to distinguish between these two types of cofibration was first discovered by Pavlov and Scholbach and pointed to the author by Harper (see \cite{HaCo15} for a discussion). In short, the fact that for $X_m \in \mathsf{S}_{\**}^{\Sigma_m} \subset \mathsf{S}_{\**}^{\Sigma}$, $m \geq 1$, then $(X_m)^{\otimes n} \in \mathsf{S}_{\**}^{\Sigma_{m n} \times \Sigma_n} \subset \mathsf{S}_{\**}^{\Sigma \times \Sigma_n}$ is only ``$\Sigma_n$-proj'' cofibrant in the first sense (compare this with the proof of Theorem \ref{SIGMANPUSHPROD THM}) is the key motivation for the $S$ $\Sigma$-inj $\Sigma_n$-proj stable model structure on $(\mathsf{Sp}^{\Sigma})^{\Sigma_n}$ introduced in this paper.
\end{remark}


\begin{remark}
Remark \ref{NATOC REM} notwithstanding, we will often produce and iterate injective and projective structures. In fact, cofibrations in any of our model structures are obtained by iterating such constructions, and the interested reader will find we often choose names accordingly. As a typical example, the $S$ $\Sigma$-inj $G$-proj cofibrations in $(\mathsf{Sp}^{\Sigma})^G$ of Theorems \ref{SIGMANPUSHPROD THM} and \ref{BIQUILLEN THM} can be built by building a 
$S$-projective structure (abbreviated to $S$ following \cite{HSS}, \cite{Sh04}) over a $\Sigma$-injective structure over a $G$-projective structure over the standard model structure in $\mathsf{S}_{\**}^{\mathbb{N}}$.
\end{remark}

\subsection{Pushout product and left Quillen bifunctors}\label{PUSHPROD SECSEC}

\begin{definition}
Consider a bifunctor $\otimes$ from categories $\mathcal{C}$, $\mathcal{D}$ to a category $\mathcal{E}$, i.e., a functor of the form
\[\mathcal{C} \times \mathcal{D} \xrightarrow{\minus \otimes \minus} \mathcal{E}.\]
Given maps $c \xrightarrow{f} \bar{c}$ in $\mathcal{C}$ and $d \xrightarrow{g} \bar{d}$ in $\mathcal{D}$, we define their \textit{pushout product} $f \square g$ (sometimes denoted $f \square^{\otimes} g$ to avoid confusion) to be the induced map
\[c \otimes \bar{d} \underset{c \otimes d}{\amalg} \bar{c} \otimes d 
\xrightarrow{f \square g} \bar{c} \otimes \bar{d}.\]
\end{definition}
For model categories $\mathcal{C}$, $\mathcal{D}$, $\mathcal{E}$ one defines the following (cf. \cite[Def. 4.2.1]{Hov98}).

\begin{definition}\label{QUILLEN BIFUNCTOR DEF}
A bifunctor 
$\mathcal{C} \times \mathcal{D} \xrightarrow{\minus \otimes \minus} \mathcal{E}$
between model categories is called a \textit{left Quillen bifunctor} if
\begin{itemize}
 \item for $c \in \mathcal{C}$ (resp. $d \in \mathcal{D}$), 
the functors $c \otimes (\minus) \colon \mathcal{D} \to \mathcal{E}$ 
(resp. $(\minus) \otimes d \colon \mathcal{C} \to \mathcal{E}$) have right adjoints;
 \item $\otimes$ satisfies the \textit{pushout product axiom}: for $f$ a cofibration in $\mathcal{C}$ and $g$ a cofibration in $\mathcal{D}$, $f \square g$ is a cofibration in $\mathcal{E}$, which is trivial if $f$ or $g$ is. 
\end{itemize}
\end{definition}

\begin{remark}\label{PUSHPRODGEN REM}
If $\mathcal{C}$, $\mathcal{D}$ are cofibrantly generated, a standard ``retract of a transfinite composition of pushouts'' argument (cf. \cite[Lemma 4.2.4]{Hov98}) shows that it suffices to check the pushout product axiom for generating (trivial) cofibrations.
\end{remark}

\begin{remark}\label{PUSHPRODADV REM}
It is immediate that if $\otimes$ is a left Quillen bifunctor then both:
\begin{enumerate*}
\item[(i)] $c \otimes (\minus)$, $(\minus)\otimes d$ are left Quillen for cofibrant $c \in \C$, $d \in \D$;
\item[(ii)] for $f$ in $\C$ and $g$ in $\D$ cofibrations between cofibrant objects, then so is $f \otimes g$.
\end{enumerate*}

However, as Remark \ref{PUSHPRODGEN REM} shows, it is technically preferable to verify the pushout product axiom rather than (i) or (ii). Indeed, that axiom is required to argue (i) via a filtration of $c$, $d$ and the analogue of Remark \ref{PUSHPRODGEN REM} fails for (ii).
\end{remark}

\section{Model structures on $G$-spectra} \label{SINJPROJMOD SEC}

In this section we build the model structures featured in Theorems \ref{SIGMANPUSHPROD THM} and \ref{BIQUILLEN THM}. 

Sections \ref{NEC SPACES} and \ref{DEFINITION} build the $S$ $\Sigma$-inj $G$-proj model structure on $(\mathsf{Sp}^{\Sigma})^G$, the new ``lax $G$-projective'' structure capturing (for $G=\Sigma_n$) the $\Sigma_n$-cofibrancy of $X^{\wedge n}$ when $X$ is positive cofibrant.
We closely follow the four model structures approach of \cite{Sh04}, Section \ref{NEC SPACES} dealing with $\mathsf{S}_{\**}^{G \times \Sigma_m}$ and Section \ref{DEFINITION} with $(\mathsf{Sp}^{\Sigma})^G$.

Section \ref{MONOS SEC} builds the auxiliary monomorphism stable and $S$ stable model structures on $(\mathsf{Sp}^{\Sigma})^G$ that appear in Theorem \ref{BIQUILLEN THM} and which, while technically novel, are just injective versions of the eponymous structures on $\mathsf{Sp}^{\Sigma}$ (cf. \cite{HSS}).

\subsection{$\Sigma$-inj $G$-proj model structure on $\mathsf{S}_{*}^{G \times \Sigma_m}$}\label{NEC SPACES}

The following is well known. Remark \ref{TWOMON REM} discusses the condition on $H$.

\begin{proposition}\label{GSPAC PROP}
For $G$ any finite group there exists a cofibrantly generated mo\-del structure on $\mathsf{S}_{\**}^{G\times \Sigma_m}$ such that weak equivalences (resp. fibrations) are the maps $A\to B$ such that $A^H\to B^H$ is a weak equivalence (resp. fibration) in $\mathsf{S}_{\**}$ for any $H \leq G \times \Sigma_m$ satisfying $H \cap G \times \** =\**$. 
Further, this is a left proper cellular simplicial model category. 
\end{proposition}

\begin{proof}
We apply the usual small object argument in \cite[Thm. 2.1.19]{Hov98} with the generating sets $I$, $J$ built from those in $\mathsf{S}_{\**}$ by inducing along each $H$. Explicitly
\[I=\underset{ H\cap G \times \**=\**}{\bigcup}\{
  (G\times \Sigma_m)/H \cdot (\partial \Delta^k_+ \to \Delta^k_+) \},\phantom{|}
  J=\underset{ H\cap G \times \**=\**}{\bigcup}\{
  (G\times \Sigma_m)/H \cdot(\Lambda^k_{l+} \to \Delta^k_+) \}.\]
Only the claim that maps in $J$-cell are weak equivalences is non obvious. 
This follows for maps in $J$ by direct calculation, for pushouts of those by Proposition \ref{FIXEDPUSH PROP} and for transfinite compositions since those commute with $(\minus)^H$.
Left properness, cellularity, and the simplicial model structure axioms are clear.
\end{proof}

\begin{remark}
Analogous model structures can be built using more general conditions on $H$, such as families of subgroups.
\end{remark}

\begin{proposition}\label{NGSPAC PROP}
For $G$ any finite group there exists a cofibrantly generated mo\-del structure on $\mathsf{S}_{\**}^{G\times \Sigma_m}$, which we call the \textbf{$\Sigma_m$-inj $G$-proj model structure}, where 
\begin{itemize}
\item cofibrations are as in the model structure in Proposition \ref{GSPAC PROP};
\item weak equivalences are the underlying weak equivalences in $\mathsf{S}_{\**}$.
\end{itemize}
Further, this is a left proper cellular simplicial model category. 
\end{proposition}

\begin{proof}
This follows by left Bousfield localization using \cite[Thm. 4.1.1]{Hi03} with respect to a suitably chosen set of maps $\mathcal{S}$. We set (cf. \cite[Prop. 1.3]{Sh04})
\[\mathcal{S}=\left\{G \times \Sigma_m\cdot_H EH_+ \to \left((G\times \Sigma_m)/H\right)_+ : H \cap G \times \** = \**\right\}\]
where $EH$ denotes a simplicial classifying space for $H$. It remains to show that the $\mathcal{S}$-equivalences are precisely the underlying weak equivalences.

Since the maps in $\mathcal{S}$ are underlying weak equivalences between cofibrant objects, \cite[Prop. 3.3.18(1)]{Hi03} applied to the forget-free power adjunction
\[\fgt \colon 
\mathsf{S}_{\**}^{G\times \Sigma_m} \rightleftarrows \mathsf{S}_{\**}\colon (\minus)^{\times (G\times \Sigma_m)}\]
yields that all $\mathcal{S}$-local equivalences are underlying weak equivalences. 

To prove the converse it suffices to show that between $\mathcal{S}$-local objects any levelwise weak equivalence is a weak equivalence in the sense of Proposition \ref{GSPAC PROP}. Since a fibrant object $X$ is $\mathcal{S}$-local precisely if one has induced weak equivalences 
\[X^H=\Map\left(\left(G\times \Sigma_m/H\right)_+,X\right)^{G \times \Sigma_m}
\xrightarrow{\sim} 
  \Map\left(G\times\Sigma_m\cdot_H EH_+,X\right)^{G \times \Sigma_m}=X^{h H},\]
the result follows due to $(\minus)^{h H}$ preserving underlying weak equivalences.
\end{proof}

\subsection{Existence of the $S$ $\Sigma$-inj $G$-proj stable model structure}\label{DEFINITION}

As in \cite[Prop. 2.2]{Sh04}, Proposition \ref{NGSPAC PROP} induces a level model structure on $(\mathsf{Sp}^{\Sigma})^G$.

\begin{proposition}\label{LGSPEC PROP}
For $G$ any finite group there exists a cofibrantly generated mo\-del structure on $(\mathsf{Sp}^{\Sigma})^G$, which we call the \textbf{$S$ $\Sigma$-inj $G$-proj level model structure}, where 
\begin{itemize}
\item weak equivalences are the maps $X\to Y$ such that $X_m\to Y_m, m\geq 0$ are underlying weak equivalences in $\mathsf{S}_{\**}$.
\item fibrations are the maps $X\to Y$ such that $X_m\to Y_m, m\geq 0$ are underlying fibrations in the $\Sigma_m$-inj $G$-proj model structure on $\mathsf{S}_{\**}^{G\times \Sigma_m}$.
\end{itemize}

Further, this is a left proper cellular simplicial model category. 
\end{proposition}

\begin{proof}
Let $I_m$ (resp. $J_m$) denote the sets of generating (resp. trivial) cofibrations for each $\Sigma_m$-inj $G$-proj model structure regarded as maps in $\mathsf{S}_{\**}^{G \times \Sigma}$. The proof of the existence of the model structure follows exactly as in \cite[Prop. 2.2]{Sh04} by setting $I=\bigcup_{m \geq 0} S \otimes I_m$ (resp. $J=\bigcup_{m \geq 0} S \otimes J_m$) as the set of generating (resp. trivial) cofibrations in $(\mathsf{Sp}^{\Sigma})^G$.
The claims of left properness, cellularity, and the simplicial model structure axioms are again straightforward.
\end{proof}

\begin{remark}\label{SIGMAGENCOF REM}
Analyzing the proofs of Propositions \ref{GSPAC PROP} and \ref{LGSPEC PROP} yields an explicit description of the generating cofibrations in Proposition \ref{LGSPEC PROP} (and Theorem \ref{SGSPEC THM})
\[I = \left\{ 
    S \otimes \left( (G\times \Sigma_m)/H \cdot 
    \left( \partial \Delta^k_+ \to \Delta^k_+  \right) \right) :
     m\geq 0, H \cap G\times \**=\**
\right\}.\]
\end{remark}

\begin{theorem}\label{SGSPEC THM}
For $G$ any finite group there exists a cofibrantly generated model structure on $(\mathsf{Sp}^{\Sigma})^G$, which we call the \textbf{$S$ $\Sigma$-inj $G$-proj stable model structure}, where 
\begin{itemize}
\item cofibrations are as in the model structure in Proposition \ref{LGSPEC PROP};
\item weak equivalences are the underlying stable equivalences in $\mathsf{Sp}^{\Sigma}$.
\end{itemize}
Further, this is a left proper cellular simplicial model category. 
\end{theorem}

\begin{proof}
This again follows by left Bousfield localization using \cite[Thm. 4.1.1]{Hi03}, this time localizing with respect to the set (cf. \cite[Thm. 2.4]{Sh04})
\[\mathcal{S}_G=\{S \otimes 
\left( G \times \Sigma_{m+1} \cdot S^1 \to G\times \Sigma_m \cdot S^0 \right) : 
  m\geq 0\}.\]
It remains to check that $\mathcal{S}_G$-equivalences coincide with stable equivalences.

For $G=\**$ this is well known since then our model structure reduces to that of \cite[Thm. 2.4]{Sh04}. 

We will reduce the general case to the case $G=\**$. To do so, start by considering the $S$ $G$-proj $\Sigma$-inj stable model structure on $(\mathsf{Sp}^\Sigma)^G$, which is built as the $G$-projective model structure over the model structure on $\mathsf{Sp}^\Sigma$ in the previous paragraph. We claim this model structure can alternatively be built by first building its level version, then localizing with respect to $\mathcal{S}_G$. Since both procedures create localizations of said level version, it suffices to check that they lead to the same local objects, and that follows since $\mathcal{S}_G=G \cdot \mathcal{S}_{\**}$.

To relate this to our intended model structure, consider the identity Quillen equivalence
\begin{equation}\label{ID QUILLEN EQ}
\Id \colon (\mathsf{Sp}^{\Sigma})^G \rightleftarrows (\mathsf{Sp}^{\Sigma})^G \colon \Id
\end{equation}
where the left hand $(\mathsf{Sp}^{\Sigma})^G$ has the 
$S$ $G$-proj $\Sigma$-inj
\textit{level} model structure and the right hand $(\mathsf{Sp}^{\Sigma})^G$ has the 
$S$ $\Sigma$-inj $G$-proj
\textit{level} model structure. 
It now suffices to check both sides have the same $\mathcal{S}_G$-local equivalences, and this is clear since mapping spaces can be simultaneously computed using cofibrant replacements in the left hand side and fibrant replacements in the right hand side.
\end{proof}

\subsection{Monomorphism stable and $S$ stable model structures}\label{MONOS SEC}

When $G=\**$, the existence of the following model structure is asserted without proof in the discussion preceding \cite[Thm. 5.3.7]{HSS}. For the sake of completeness (and as a warm-up to Theorem \ref{OINJMODEL THM}), we include a proof sketch combining arguments of \cite[Sec. 5]{HSS} with the localization machinery of \cite[Thm. 4.1.1]{Hi03}.

\begin{theorem}\label{GSTABLE THM}
For $G$ any finite group there exists a cofibrantly generated model structure on $(\mathsf{Sp}^{\Sigma})^G$, which we call the {\bf monomorphism stable model structure}, where
\begin{itemize}
\item cofibrations are the maps $X\to Y$ such that $X_m\to Y_m$ is a monomorphism of pointed simplicial sets for each $m \geq 0$.
\item weak equivalences are the underlying stable equivalences in $\mathsf{Sp}^{\Sigma}$.
\end{itemize}
Further, this is a left proper cellular simplicial model category.
\end{theorem}

\begin{proof} We start by building the analogue level weak equivalence model structure. When $G=\**$, this is precisely the injective level structure in \cite[Thm. 5.1.2]{HSS}, and the interested reader can check that the somewhat lengthy proof there generalizes. Instead, we point out that much of the argument can be streamlined by instead verifying the conditions in \cite[Thm. 2.1.19]{Hov98}.

Setting $I$ (resp. $J$) to be a set of representatives of monomorphisms (resp. monomorphisms that are level weak equivalences) between countable $G$-spectra (cf. proof of \cite[Thm. 5.1.2]{HSS}), parts 1,2,3 of \cite[Thm. 2.1.19]{Hov98} are immediate, part 4 follows since $J \subset I$ and colimits are levelwise and part 5 follows by noting that $I$ contains the maps of the form 
$S \otimes \left( G \times \Sigma_m  \cdot (\partial \Delta^k_+\to \Delta^k_+)\right)$, 
so that $I$-inj consists of level equivalences. For the harder part 6, one needs to show a lift exists in any diagram  
\[\xymatrix@1{
   A \ar[r] \ar[d]_f & X \ar[d]^g \\ 
	 B \ar[r] \ar@{-->}[ru]    & Y 
}\]
with $f \in \mathcal{W} \cap I$-cof and $g \in J$-inj. This generalizes \cite[Lemma 5.14(6)]{HSS}, the proof of which applies without change once one generalizes \cite[Lemma 5.17]{HSS} to $G$-spectra. The latter can be done by simply choosing the $FC$ subspectra in the proof of \cite[Lemma 5.17]{HSS} to be  $G$-subspectra, finishing the existence argument for the level model structure. Left properness, cellularity and the simplicial model structure axioms are again straightforward.

To produce the stable version, one again applies \cite[Thm. 4.1.1]{Hi03} to the set $\mathcal{S}_G$ in the proof of Theorem \ref{SGSPEC THM}, showing that the weak equivalences are as described by arguing as in the last paragraph of the proof of that theorem.
\end{proof}

\begin{theorem}\label{SUSPEC THM}
For $G$ any finite group there exists a cofibrantly generated model structure on $(\mathsf{Sp}^{\Sigma})^G$, which we call the \textbf{$S$ stable model structure}, where 
\begin{itemize}
\item a set of generating cofibrations is  
\[I = \left\{ 
    S \otimes \left( (G\times \Sigma_m)/H \cdot 
    \left( \partial \Delta^k_+ \to \Delta^k_+  \right) \right) :
     m\geq 0, \text{ any }H \leq G \times \Sigma_m 
\right\}.\]
\item weak equivalences are the underlying stable equivalences in $\mathsf{Sp}^{\Sigma}$.
\end{itemize}
Further, this is a left proper cellular simplicial model category.
\end{theorem}

\begin{proof}
This is an analogue of Theorem \ref{SGSPEC THM}, now without conditions on $H$. The same proof, starting with analogues of Propositions \ref{GSPAC PROP}, \ref{NGSPAC PROP} and \ref{LGSPEC PROP}, applies.
\end{proof}

\begin{proposition}\label{INJECTMOD PROP}
The $S$ stable model structure on $(\mathsf{Sp}^{\Sigma})^G$ is the injective model structure over the $S$ stable model structure on $\mathsf{Sp}^{\Sigma}$.

More explicitly, the $S$ stable cofibrations (resp. weak equivalences) in $(\mathsf{Sp}^{\Sigma})^G$ are the underlying $S$ stable cofibrations (resp. weak equivalences) in $\mathsf{Sp}^{\Sigma}$.
\end{proposition}

To prove Proposition \ref{INJECTMOD PROP} we start by recalling a well known inductive procedure to build maps of spectra (cf. \cite[II.5]{Sch97}, \cite[Sec. 5.2]{HSS}).

\begin{definition}
Define $\bar{S} \in \mathsf{Sp}^{\Sigma}$ by $\bar{S}_0=\**$, $\bar{S}_m=S^m$ together with the obvious structure maps and let $i\colon \bar{S} \to S$ be the inclusion. 
For $A \in \mathsf{Sp}^{\Sigma}$, define its \textit{$m$-th latching object} to be $L_m A=(\bar{S} \wedge A)_m$ for $m \geq 0$. Note that $i$ induces a \textit{$m$-th latching map}
\[
L_m A \xrightarrow{l_m A} A_m.
\]
\end{definition}

Given spectra $A,B$ define a \textit{map up to degree $m$ from $A$ to $B$} to be a list of maps 
$\{f_{\bar{m}}\colon A_{\bar{m}} \to B_{\bar{m}}\}_{0 \leq \bar{m} \leq m}$ compatible with the spectra structure maps up to degree $m$.
The importance of latching maps comes from the following result (used implicitly in \cite[Sec. 5.2.2]{HSS}. Also, compare with \cite[Obs. 3.9]{RiVe13}).

\begin{lemma}\label{INDLIFT LEM}
A map $\{f_{\bar{m}}\colon A_{\bar{m}} \to B_{\bar{m}}\}_{0 \leq \bar{m} \leq m-1}$ up to degree $m-1$ naturally induces a map $L_{m}A \to L_{m}B$. Further, extensions to a map $\{f_{\bar{m}}\colon A_{\bar{m}} \to B_{\bar{m}}\}_{0 \leq \bar{m} \leq m}$ up to degree $m$ are in natural bijection with dashed arrows
\[
  \xymatrix@1{
      L_m A \ar[r] \ar[d]_{l_m A} & L_m B \ar[d]^{l_m B}  \\
			A_m.\ar@{-->}[r]         &  B_m.                
	}\]
\end{lemma}

\begin{remark}
By naturality Lemma \ref{INDLIFT LEM} generalizes to $(\mathsf{Sp}^{\Sigma})^G$.
\end{remark}

Proposition \ref{INJECTMOD PROP} will follow from the following analogue of \cite[Sec. 5.2.2]{HSS}.

\begin{proposition}\label{LATCHMON PROP}
The $S$ stable cofibrations in $(\mathsf{Sp}^{\Sigma})^G$ are those maps $f \colon A \to B$ such that $(f \square i)_m \colon A_m \vee_{L_m A}L_m B\to B_m$ is a monomorphism for all $m \geq 0$.
\end{proposition}

\begin{proof}
$X \to Y$ is a $S$ stable trivial fibration iff $X_m^H \to Y_m^H$ are trivial fibrations in $\mathsf{S}_{\**}$ for all $m \geq 0$, $H \leq G \times \Sigma_m$, i.e. iff $X_m \to Y_m$ are genuine $G \times \Sigma_m$ fibrations for $m \geq 0$. By Lemma \ref{INDLIFT LEM}, building a lift in the left hand diagram
\[\xymatrix@1{
   A \ar[d]_f \ar[r]                                       & X \ar[d]   & &
	 A_m \vee_{L_m A}L_m B \ar[r] \ar[d]_{(f \square i)_m}     & X_m \ar[d] \\
	 B   \ar[r]  \ar@{-->}[ru]                               & Y          & &
	 B_m \ar[r] \ar@{-->}[ru]                                & Y_m
}\] 
is the same as building successive lifts in the right hand diagrams for $m \geq 0$. Since monomorphisms have the left lifting property against genuine fibrations, the given condition is sufficient.

For the converse, by \cite[Lemma 4.2.4]{Hov98} it suffices to check $(f \square i)_m$ is a monomorphism when $f$ is a generating cofibration. Letting
\[ f = S \otimes ((G\times \Sigma_m)/H \cdot \partial \Delta^k_+ 
     \xrightarrow{f'} 
	(G\times \Sigma_m)/H \cdot \Delta^k_+ )\]
one has $f\square i=f'\square^{\otimes}i$ (where $\square^{\otimes}$ denotes the pushout product with respect to the bifunctor $\mathsf{S}_{\**}^{\Sigma}\times \mathsf{Sp}^{\Sigma} \xrightarrow{\otimes} \mathsf{Sp}^{\Sigma}$) so that the result is now clear.
\end{proof}

\begin{proof}[Proof of Proposition \ref{INJECTMOD PROP}] Cofibrations are underlying since forgetting the $G$-action does not change the characterization in Proposition \ref{LATCHMON PROP}. 
The case of weak equivalences is obvious.  
\end{proof}

\begin{remark}\label{HSIGMAINJGPROJ RMK}
In Section \ref{POSSYM SEC} we will need the $\Sigma_r$-injective model structure on $(\mathsf{Sp}^{\Sigma})^{G \times \Sigma_r}$ with regard to the 
$S$ $\Sigma$-inj $G$-proj model structure on $(\mathsf{Sp}^{\Sigma})^{G}$.
We call this the $S$ $\Sigma \times \Sigma_r$-inj $G$-proj stable model structure, and build it just as in Theorems \ref{SGSPEC THM} and \ref{SUSPEC THM} using as generating cofibrations
\[I = \left\{ 
    S \otimes \left( (G \times \Sigma_m \times \Sigma_r)/H \cdot 
    \left( \partial \Delta^k_+ \to \Delta^k_+  \right) \right) :
     m,r\geq 0, H \cap G \times \** \times \**=\**
\right\}.\]
The analogue of Proposition \ref{INJECTMOD PROP} proving $\Sigma_r$-injectiveness is shown in the same way by noting that $X \to Y$ is a trivial fibration iff $X_{m}^H \to Y_m^H$, $H \cap G \times \** \times \** = \**$ is a trivial fibration, so that $A\to B$ is a cofibration iff $(f \square i)_m \colon A_m \vee_{L_m A}L_m B\to B_m$ is built only out of simplices with isotropies $H$ satisfying $H \cap G \times \** \times \** = \**$.
\end{remark}

\section{Properties of $S$ $\Sigma$-inj $G$-proj cofibrations}\label{PROPERTIES SEC}

In this section we prove the key properties of $S$ $\Sigma$-inj $G$-proj cofibrations.

Subsection \ref{GPROJ TYPE} deals with those properties one would expect from genuine $G$-projective cofibrations, namely the ``change of group'' Propositions
\ref{PROPSMASH} and \ref{PROPINDUCE} as well as Theorem \ref{BIQUILLEN THM}.

Subsection \ref{QSECTION}, the technical heart of the paper, deals with the somewhat lengthier proof of Theorem \ref{SIGMANPUSHPROD THM}.

\subsection{$G$-projective type properties}\label{GPROJ TYPE}

\begin{proposition}\label{PROPSMASH} 
Suppose each category is equipped with its respective  $S$ $\Sigma$-inj $G$-proj stable model structure. Then the functor
\[(\mathsf{Sp}^{\Sigma})^G\times (\mathsf{Sp}^{\Sigma})^{\bar{G}} \xrightarrow{\minus \wedge \minus} (\mathsf{Sp}^{\Sigma})^{G \times \bar{G}} \]
is a left Quillen bifunctor.
\end{proposition}

\begin{proof}
The existence of the right adjoints is formal. It suffices to check the pushout product axiom (cf. Definition \ref{QUILLEN BIFUNCTOR DEF}) between generating (trivial) cofibrations (cf. Remark \ref{PUSHPRODGEN REM}) and letting (cf. Remark \ref{SIGMAGENCOF REM})
\[f= S \otimes \left( (G \times \Sigma_m)/H \cdot \left( \partial \Delta^{k}_+  \to  \Delta^k_+  \right) \right), \quad
g= S \otimes \left( (\bar{G} \times \Sigma_{\bar{m}})/\bar{H} \cdot \left( \partial \Delta^{\bar{k}}_+  \to  \Delta^{\bar{k}}_+  \right) \right)
\]
one has (using the identification $H\times \bar{H} \subset G \times \Sigma_m\times \bar{G} \times \Sigma_{\bar{m}} \subset G \times \bar{G} \times \Sigma_{m+\bar{m}}$)
\[f \square g =
 S \otimes \left( (G\times \bar{G}\times \Sigma_{m+\bar{m}})/(H\times \bar{H}) \cdot \left((\partial (\Delta^k\times \Delta^{\bar{k}}))_+\to  (\Delta^k\times \Delta^{\bar{k}})_+ \right) \right),\]
which is a cofibration since Remark \ref{SIGMAGENCOF REM} implies $H\times \bar{H} \cap G\times \bar{G}\times \{\**\}=\**$.

The extra claim that $f \square g$ is a weak equivalence if either $f$ or $g$ is can be checked by forgetting the actions of $G,\bar{G}$, reducing to \cite[Thm. 5.3.7(5)]{HSS}.
\end{proof}

\begin{proposition}\label{PROPINDUCE}
Let $\bar{G}\subset G$ be finite groups, and suppose each category is equipped with its respective $S$ $\Sigma$-inj $G$-proj stable model structure. Then both 
\[\fgt\colon(\mathsf{Sp}^\Sigma)^G \rightleftarrows
 (\mathsf{Sp}^\Sigma)^{\bar{G}}\colon ((\minus)^{G \cdot S})^{\bar{G}} \quad \text{ and }\quad 
  G\cdot_{\bar{G}} (\minus) \colon (\mathsf{Sp}^\Sigma)^{\bar{G}} \rightleftarrows (\mathsf{Sp}^\Sigma)^G \colon \fgt\]
are Quillen adjunctions.
\end{proposition}

\begin{proof}
This is immediate for the first adjunction since $\fgt$ preserves weak equivalences and free actions.
For the second one, choose a generating cofibration
\[f = S \otimes \left((\bar{G}\times \Sigma_m)/H \cdot \left(\partial \Delta^{k}_+  \to  \Delta^k_+ \right) \right) \]
so that
\[G\cdot_{\bar{G}}f = S \otimes \left((G\times \Sigma_m)/H \cdot \left(\partial \Delta^{k}_+  \to  \Delta^k_+ \right) \right)\]
which is again a cofibration since $H \cap \bar{G}\times \{\**\}=\**$ implies  $H \cap G\times \{\**\}=\**$.

That $G \cdot_{\bar{G}} (\minus)$ applied to a trivial cofibration yields a weak equivalence follows by forgetting the actions since then $G \cdot_{\bar{G}} (\minus)$ is a wedge over $G/\bar{G}$.
\end{proof}

We now turn to the proof of Theorem \ref{BIQUILLEN THM}. We will make use of the following analogue for bifunctors of the ``universal property of left Bousfield localizations'' in \cite[Prop. 3.3.18(1)]{Hi03}.

\begin{lemma}\label{LOCBIADJ LEM}
Suppose 
\[\C \times \D \xrightarrow{\minus \otimes \minus } \mathcal{E}\]
is a left Quillen bifunctor, that $\S$ is a class of maps between cofibrant objects of $\C$ such that the left Bousfield localization $L_{\S}\C$ exists, and that $\D$ is a cofibrantly generated model category for which the generating cofibrations have cofibrant domains and codomains. Then (recall that as categories $L_{\S}\C=\C$)
\[ L_{\S}\C \times \D \xrightarrow{\minus \otimes \minus}  \mathcal{E}\]
remains a left Quillen bifunctor iff $f \otimes d$ is a weak equivalence in $\mathcal{E}$ for each $f \in \S$ and $d$ a domain or codomain of a generating cofibration of $\D$.
\end{lemma}

\begin{proof}
First note that by Remark \ref{PUSHPRODGEN REM} $\otimes$ will remain a left Quillen bifunctor precisely if the pushout product axiom holds when $f \colon c \to \bar{c}$ is a \textit{trivial} cofibration in $L_{\S} \C$ and $g \colon d \to \bar{d}$ is a \textit{generating} cofibration in $\D$.
Since $(\minus) \otimes d$, $(\minus) \otimes \bar{d}$ are left Quillen functors with respect to the original model structure $\C$, \cite[Prop. 3.3.18(1)]{Hi03} shows that the condition in the theorem is necessary and that, if that condition holds, the horizontal maps in
\[\xymatrix@1{
    c \otimes d \ar@{ >->}[r]^{\sim} \ar[d] & \bar{c} \otimes d \ar[d] \\
		c \otimes \bar{d} \ar@{ >->}[r]^{\sim} & \bar{c} \otimes \bar{d}}\]
are trivial cofibrations. The 2-out-of-3 property now implies that $f \square g$ is a weak equivalence, showing that the condition in the theorem is also sufficient.
\end{proof}

\begin{proof}[Proof of Theorem \ref{BIQUILLEN THM}]
The existence of the required right adjoints is formal.

We will prove the remainder of both parts in parallel.

As a first step we prove the analogue result with stable structures replaced by level structures throughout. Since the generating (trivial) cofibrations in the $S$ $\Sigma$-inj $G$-proj level model structure all have the form $S\otimes f$ for some $f$ in $(\mathsf{S}_{\**})^G$ (cf. Remark \ref{SIGMAGENCOF REM}), this reduces to showing the analogue result for the bifunctor
\[\mathsf{S}_{\**}^{G\times \Sigma} \times (\mathsf{Sp}^{\Sigma})^G \xrightarrow{\minus \otimes_{G} \minus} \mathsf{Sp}^{\Sigma},\]
where $\mathsf{S}_{\**}^{G\times \Sigma}$ has the $\Sigma$-inj $G$-proj model structure obtained by combining the $\Sigma_m$-inj $G$-proj model structures of Proposition \ref{NGSPAC PROP} for each $m \geq 0$.

For the monomorphism case, choose (cf. Proposition \ref{GSPAC PROP}) a generating (resp. trivial) cofibration in $\mathsf{S}_{\**}^{G\times \Sigma}$
\[f =  (G\times \Sigma_m)/H \cdot f'\]
($f'$ a generating (resp. trivial) cofibration in 
$\mathsf{S}_{\**}$, $H \cap G \times \** = \**$) and a monomorphism $g$ in $(Sp^{\Sigma})^G$.
Then, using the identification $\Sigma_{m} \times \Sigma_{\bar{m}-m} \subset \Sigma_{\bar{m}}$,
\[\begin{aligned}
(f\square^{\otimes_G}g)_{\bar{m}}\simeq &
  \left((f\square^{\otimes}g)_{\bar{m}}\right)_G \simeq 
  \left(G \times \Sigma_{\bar{m}}\underset{H \times \Sigma_{\bar{m}-m}}{\cdot}
	f'\square^{\wedge}g_{\bar{m}-m}\right)_G \simeq \\
  \simeq & \left(G \backslash (G \times \Sigma_{\bar{m}}) /(H \times \Sigma_{\bar{m}-m})\right)	\cdot
	f'\square^{\wedge}g_{\bar{m}-m},
\end{aligned}\]
where the last step follows since the condition $H \cap G \times \** = \**$ implies $G$ acts freely on cosets $(G \times \Sigma_{\bar{m}}) /(H \times \Sigma_{\bar{m}-m})$. It is now clear that $f\square^{\otimes_G}g$ is a monomorphism, level trivial if either $f'$ or all $g_m$ are.

For the $S$ level case, note first that by the monomorphism case we need no longer worry about trivial cofibrations. For the case of regular cofibrations, choose generating cofibrations $f$ in $\mathsf{S}_{\**}^{G\times \Sigma}$ as above and
\[g = S \otimes \left( (G\times \Sigma_{\bar{m}})/\bar{H} \cdot g'\right)\]
($g'$ a generating cofibration in $\mathsf{S}_{\**}$, any $\bar{H} \leq G \times \Sigma_{\bar{m}}$) in $(\mathsf{Sp}^{\Sigma})^G$. Then, using the identification $\Sigma_{m} \times \Sigma_{\bar{m}} \subset \Sigma_{m+\bar{m}}$,
\[f \square^{\otimes_G}g=\left(f \square^{\otimes} g\right)_G = 
  S \otimes \left( G \backslash (G\times G\times \Sigma_{m+\bar{m}})/(H\times \bar{H}) \cdot f'\square g' \right)
\]
which is indeed a $S$ cofibration, finishing the proof of the analogue level result.

We now turn to the second step, showing that $\wedge_G$ remains a left Quillen bifunctor after stabilizing the model structures. In all cases one is localizing by $\mathcal{S}_G$  (cf. proofs of Theorems \ref{SGSPEC THM}, \ref{GSTABLE THM}, \ref{SUSPEC THM}), and hence by Lemma \ref{LOCBIADJ LEM} it suffices to verify $f \otimes_G A$ is a stable equivalence for $f \in \mathcal{S}_G$ and $A$ a suitably cofibrant $G$-spectrum. It suffices to deal with the case of monomorphism cofibrant $A$ (i.e., any $A$), and since $\mathcal{S}_G= G \cdot \mathcal{S}_{\**}$ this reduces to the case $G=\**$. But for $G=\**$ the claim follows by \cite[Thm. 5.3.7(5)]{HSS}, finishing the proof.
\end{proof}

\subsection{Lax $\Sigma_n$-cofibrancy of $n$-fold pushout products}\label{QSECTION} 


In this section we prove Theorem \ref{SIGMANPUSHPROD THM}. Roughly speaking, the proof will follow by induction using the usual ``retract of a transfinite composition of pushouts of generating cofibrations'' description of cofibrations. The main obstacle is the fact that the $n$-fold pushout product $\square^n$ does not respect compositions of maps. Handling those will require two key technical results, Lemmas \ref{MAIN428} and \ref{COFSQUARE}. 

To prove Lemmas \ref{MAIN428}, \ref{COFSQUARE} and Theorem \ref{SIGMANPUSHPROD THM} we will need some notation.

\begin{definition}\label{DIAGRAMQ}
Let $I$ be a poset and $i \colon I \to \mathsf{Sp}^{\Sigma}$. We denote by $i^{\wedge n}$ the ``cubical'' diagram 
\[ i^{\wedge n} \colon I^{\times n} \xrightarrow{i^{\times n}} (\mathsf{Sp}^{\Sigma})^{\times n} \xrightarrow{\wedge} \mathsf{Sp}^{\Sigma}.\]
Further, for $T \subset I^{\times n}$ any subset, we denote $Q^n_T(i) = \colim_T (i^{\wedge n})$.
Note that when $T$ is closed under the obvious $\Sigma_n$-action on $I^{\times n}$ one obtains an induced $\Sigma_n$-action on $Q^n_T(i)$.
\end{definition}

\begin{remark}\label{QNT RMK}
Borrowing from \cite{Ha09}, we let $Q^n_t(i)$ denote 
 $Q^n_{T_t}(i)$, where $i=X\to Y$ is viewed as a functor $(0 \to 1) \to \mathsf{Sp}^{\Sigma}$ and $T_t$ is the subset of $(0 \to 1)^{\times n}$ of those tuples with at most $t$ $1$-entries.
\end{remark}

The objects $Q^n_T(i)$ are related to latching objects/maps (cf. \cite[Obs. 3.8]{RiVe13}).

\begin{definition}
Given $e \in I^{\times n}$ set $T^n_{e}=\{\bar{e} \in I^{\times n} : \bar{e} < e\}$. Further, given $i \colon I \to \mathsf{Sp}^{\Sigma}$, define the \textit{latching map of $i^{\wedge n}$ at $e$} as the natural map
\[L_e^n(i^{\wedge n})=Q^n_{T_e}(i) \xrightarrow{l_e^n(i^{\wedge n})} i^{\wedge n}(e).\]
\end{definition}

A straightforward computation reveals the following relationship between latching maps and the pushout product (cf. \cite[Example 4.6]{RiVe13}).

\begin{proposition}\label{LATCHPUSH PROP}
Let $e_1 \in I^{\times n_1}$ and $e_2 \in I^{\times n_2}$, so that $(e_1,e_2) \in I^{\times (n_1+n_2)}$. Then
\[l^{n_1+n_2}_{(e_1,e_2)}\left(i^{\wedge(n_1+n_2)}\right) =
 l_{e_1}^{n_1}\left(i^{\wedge n_1}\right) \square l_{e_2}^{n_2}\left(i^{\wedge n_2}\right).\]
\end{proposition}

The following is the key technical lemma in this section. The proof of this result, which is essentially lifted from the appendix to the author's thesis\footnote{
That appendix proved the analogue claim for $\Sigma_{n}$-projective cofibrations, an ultimately useless fact since the $\Sigma_{n}$-projective analogue of Theorem \ref{SIGMANPUSHPROD THM} fails for generating cofibrations.
}, 
explores generalizations of filtrations found in \cite[Sec. 12]{ElmMan06}, \cite[Def. 4.13]{Ha09} from single maps to compositions of maps. A similar result, with modified hypotheses and conclusions but sharing some of the key ideas in the proof, was proven independently by David White in \cite{Wh14}.

\begin{lemma}\label{MAIN428}
Let $i \colon (0\to 1 \to 2) \to \mathsf{Sp}^{\Sigma}$ be a diagram $Z_0\xrightarrow{f_1} Z_1 \xrightarrow{f_2} Z_2$ such that  
\[f_i^{\square \bar{n}}\colon Q^{\bar{n}}_{\bar{n}-1}(f_i)\to Z_i^{\wedge \bar{n}} ,\quad 0\leq \bar{n} \leq n, i=1,2\]
are $S$ $\Sigma$-inj $\Sigma_{\bar{n}}$-proj cofibrations 
in $(\mathsf{Sp}^\Sigma)^{\Sigma_{\bar{n}}}$.

Choose $T\subset \bar{T} \subset (0\to 1 \to 2)^{\times n}$ symmetric convex (recall $T$ is called convex if $e \in T$ and $\bar{e} \leq e$ implies $\bar{e} \in T$) subsets containing any tuple that has at least one $0$-entry. Then the map
\[Q^n_T(i) \to Q^n_{\bar{T}}(i)\]
is a $S$ $\Sigma$-inj $\Sigma_{n}$-proj stable cofibration.

Additionally, if one also knows that $Z_0^{\wedge \bar{n}}, 0\leq \bar{n} \leq n$ is $S$ $\Sigma$-inj $\Sigma_{\bar{n}}$-proj cofibrant then the conclusion above holds for any symmetric convex $T\subset \bar{T}$.
\end{lemma}

\begin{proof}
We deal with the main and additional cases in parallel. 

Without loss of generality we assume $\bar{T}$ is obtained from $T$ by adding the orbit of some  
$e = (e_0,e_1,e_2) \in  \{0\}^{\times n_0}  \times \{1\}^{\times n_1} \times \{2\}^{\times n_2}$. Then $T_e^n \subset T$ and one has a pushout diagram
\[\xymatrix@1{
   \Sigma_n \underset{\Sigma_{n_0} \times \Sigma_{n_1}\times \Sigma_{n_2}}{\cdot}
    Q^n_{T_e}(i) \ar[r] \ar[d]_-{\Sigma_n \underset{\Sigma_{n_0} \times \Sigma_{n_1}\times \Sigma_{n_2}}{\cdot}l_e^n(i^{\wedge n})}
		& Q^n_T(i) \ar[d] \\
   \Sigma_n \underset{\Sigma_{n_0} \times \Sigma_{n_1}\times \Sigma_{n_2}}{\cdot}
	  Z_0^{n_0}\wedge Z_1^{n_1}\wedge Z_2^{n_2} \ar[r] 
		& Q^n_{\bar{T}}(i),
}\]
so that it suffices to show that the left hand map is a $S$ $\Sigma$-inj $\Sigma_{n}$-proj cofibration, and by Proposition \ref{PROPINDUCE} this reduces to showing that the latching map $l_e^n(i^{\wedge n})$ is a $S$ $\Sigma$-inj $\Sigma_{n_0} \times \Sigma_{n_1}\times \Sigma_{n_2}$-proj cofibration.  
Proposition \ref{LATCHPUSH PROP} then identifies
\[l^{n}_e\left(i^{\wedge n}\right)=
  l^{n_0}_{e_0}(i^{\wedge n_0}) \square l^{n_1}_{e_1}(i^{\wedge n_1}) \square l^{n_2}_{e_2}(i^{\wedge n_2})=
  Z_0^{\wedge n_0} \wedge f_1^{\square n_1} \square f_2^{\square n_2}
\]
(for the identification $l^{n_2}_{e_2}(i^{\wedge n_2})=f_2^{\square n_2}$, note that the tuples without 0-entries are final in $T_{e_2}^{n_2}\subset (0 \to 1 \to 2)^{\times n_2}$).
Now note that in the main case $T$ already contains all tuples with a $0$-entry so that it must be $n_0=0$, while in the additional case $n_0$ can take any value. In either case Proposition \ref{PROPSMASH} finishes the proof.
\end{proof}

\begin{remark}
While it is straightforward to generalize Lemma \ref{MAIN428} to longer compositions (of three or more maps), such generalizations will not be necessary.
\end{remark}

\begin{lemma}\label{COFSQUARE}
Let $Z_0\xrightarrow{f_1} Z_1 \xrightarrow{f_2} Z_2$ be as in Lemma \ref{MAIN428}.

If one knows additionally that $Z_0^{\wedge \bar{n}}, 0\leq \bar{n} \leq n$ is $S$ $\Sigma$-inj $\Sigma_{\bar{n}}$-proj  cofibrant, then the maps (where the $Q^n_{\bar{n}}$ objects are defined in Remark \ref{QNT RMK})
$$Q^n_{\bar{n}}(f_2f_1)\bigvee_{Q^{n}_{\bar{n}}(f_1)} Q^n_{\bar{n}+1}(f_1) \to  Q^n_{\bar{n}+1}(f_2f_1),\quad 0\leq \bar{n} < n$$
are $S$ $\Sigma$-inj $\Sigma_{n}$-proj stable cofibrations.

Further, absent the additional condition, the result still holds when $\bar{n}=n-1$.
\end{lemma}

\begin{proof}
This is a direct consequence of Lemma \ref{MAIN428} by identifying all objects with $Q^n_{T}(i)$ for some $T$. For $Q^n_{k}(f_1)$ this is $T^1_k$, the subset of tuples with no $2$-entries and at most $k$ $1$-entries, while for $Q^n_{k}(f_2f_1)$ it is $T^2_k$, the subset of tuples with at least $n-k$ $0$-entries (or equivalently, at most $k$ $2$-or-$1$-entries). The result then follows by noting that $ T^2_{\bar{n}} \cap T^1_{\bar{n}+1} = T^1_{\bar{n}}$ and $ T^2_{\bar{n}} \cup T^1_{\bar{n}+1} \subset T^2_{\bar{n}+1}$.
\end{proof}

All we are now missing to prove Theorem \ref{SIGMANPUSHPROD THM} is the following lemma, which handles the pushout case (compare with \cite[Prop. 6.13]{Ha09}).  

\begin{lemma}\label{PUSH LEM}
Consider a pushout diagram 
\begin{equation}\label{SQUARE}
\xymatrix@1{A \ar[r]\ar[d]_i & C \ar[d]^f \\
          B \ar[r]         & D.}
\end{equation}
If $i^{\square n}$ is a (trivial) $S$ $\Sigma$-inj $\Sigma_n$-proj cofibration in $(\mathsf{Sp}^\Sigma)^{\Sigma_n}$ then so is $f^{\square n}$.
\end{lemma}

\begin{proof}\label{BASEQ}
It suffices to show that
\[\xymatrix@1{
 Q_{n-1}^n(i) \ar[r]\ar[d]_{i^{\square n}} & Q^n_{n-1} (f) \ar[d]^{f^{\square n}}  \\
 B^{\wedge n} \ar[r]                       & D^{\wedge n} }\]
is itself a pushout diagram. This is \cite[Prop. 6.13]{Ha09}, where it is left as an exercise. Alternatively, note that the pushout product 
$ \square $ is a bifunctor between arrow categories (cf. \cite[Def. 4.4.]{RiVe13}) which takes pushout diagrams in each arrow variable to pushout diagrams, so that the result follows by considering the arrow category diagram 
$i^{\square n} \to f \square i^{\square (n-1)}\to 
 f^{\square 2} \square i^{\square (n-2)} \to \cdots \to f^{\square n}$.
\end{proof}

We now prove Theorem \ref{SIGMANPUSHPROD THM}. The proof is similar to that of \cite[Prop. 4.28*]{HaCo15} (also, compare \cite[Prop. 4.28]{Ha09}), except now boosted by Lemma \ref{COFSQUARE}.

\begin{proof}[Proof of Theorem \ref{SIGMANPUSHPROD THM}]
Since weak equivalences ignore the $\Sigma_n$-action and the $S$ stable model structure on $\mathsf{Sp}^{\Sigma}$ is monoidal (cf. \cite[Thm. 5.5.1]{HSS}), we need not worry about trivial cofibrations. 

We argue by induction on a description of a positive $S$ cofibration $f$ in $\mathsf{Sp}^{\Sigma}$ as a retract of a transfinite composition of pushouts of generating cofibrations.

The base case is that of a generating cofibration 
$f = S \otimes ( \Sigma_m/H \cdot (\partial \Delta^k_+  \to \Delta^k_+) )$ for some $m \geq 1$, 
$H \leq \Sigma_m$. Then, using the identifications $H^{\times n}\subset (\Sigma_m)^{\times n} \subset \Sigma_{m n}$,
\[
f^{\square n}= S \otimes \left( \Sigma_{m n}/H^{\times n} \cdot 
\left(\partial( \Delta^k)^{\times n}_+  \to (\Delta^k)^{\times n}_+ \right) \right),
\]
which is a $S$ $\Sigma$-inj $\Sigma_n$-proj cofibration since the condition $m \geq 1$ implies the map
\[\left( \Sigma_{m n}/H^{\times n} \cdot \left(\partial( \Delta^k)^{\times n}_+  \to (\Delta^k)^{\times n}_+ \right) \right)\] 
is built by adding only $\Sigma_n$-free simplices.

We now move to the general case. As usual, retracts cause no difficulty, and we hence focus on a transfinite composition 
\begin{equation}\label{TRANSFINITE EQ}
A_0\xrightarrow{f_0} A_1 \xrightarrow{f_1} A_2 \xrightarrow{f_2} A_3 \xrightarrow{f_3} \dots \to A_{\kappa}=\colim_{\beta <\kappa} A_{\beta}
\end{equation}
(we use the convention $A_{\beta}=\colim_{\gamma < \beta} A_{\gamma}$ for each limit ordinal $\beta<\kappa$)
where each $f_{\beta}\colon A_{\beta} \to A_{\beta +1}$ is the pushout of a generating positive $S$  cofibration $i_{\beta}$. Further, for $\beta \leq \kappa$, 
denote by 
$\bar{f}_{\beta} \colon A_0 \to A_{\beta}$ 
the full composite of $\{f_{\gamma}\}_{\gamma<\beta}$.
 Since the $Q^n_t$ constructions preserve filtered colimits (since so does $ \wedge $ in each variable), the main claim will follow if the vertical map of $\kappa$-diagrams ($\kappa$-th map excluded)
\begin{equation}\label{KAPPADIAG EQ}
\xymatrix@1{Q^n_{n-1}(f_0)\ar[r] \ar[d] & Q^n_{n-1}(f_1f_0) \ar[r] \ar[d] & Q^n_{n-1}(f_2f_1f_0) \ar[r] \ar[d] & \dotso \ar[r] &
 Q^n_{n-1}(\bar{f}_{\kappa}) \ar[d] \\
A_1^{\wedge n} \ar[r] & A_2^{\wedge n} \ar[r] & A_3^{\wedge n} \ar[r] & \dotso \ar[r] & A_{\kappa}^{\wedge n}
} 
\end{equation}
is a $\kappa$-projective cofibration between $\kappa$-diagrams with respect to the underlying $S$ $\Sigma$-inj $\Sigma_n$-proj model structure.
One thus reduces to 
\textit{inductively} checking that the relative latching maps 
$A_{\beta}^{\wedge n} \vee_{Q_{n-1}^n(\bar{f}_{\beta})}Q_{n-1}^n(f_{\beta}\bar{f}_{\beta})
\to A_{\beta + 1}^{\wedge n}$
for successor ordinals $\beta+1$ are $S$ $\Sigma$-inj $\Sigma_n$-proj cofibrations (note that $Q^n_{n-1}(\bar{f}_0)=A_0^{\wedge n}$, so that this covers the leftmost map in (\ref{KAPPADIAG EQ}), and that latching conditions for limit ordinals are trivial). This now follows by applying Lemma \ref{PUSH LEM} to $i_{\beta}$, $f_{\beta}$ and Lemma \ref{COFSQUARE} to 
$A_0 \xrightarrow{\bar{f}_{\beta}}A_{\beta} \xrightarrow{f_{\beta}} A_{\beta +1}$ (note that $\bar{f}_{\beta}^{\square n},n \geq 0$ is a $S$ $\Sigma$-inj $\Sigma_n$-proj cofibration by the induction hypothesis), finishing the proof of the main claim.

For the extra claim, note that applying the main claim of the result to the map $\** \to A$ yields that $A^{\wedge \bar{n}}, \bar{n} \geq 0$ is $S$ $\Sigma$-inj $\Sigma_{\bar{n}}$-proj cofibrant (since $Q^{\bar{n}}_{\bar{n}-1}(* \to A)=\**$). The additional conditions in Lemma \ref{COFSQUARE} are hence satisfied and the strengthened conclusions now allow us to conclude the $\kappa$-cofibrancy for $0\leq \bar{n}<n$ of the vertical $\kappa$-diagram map ($\kappa$-th map excluded)
\[
\xymatrix@1{
  Q^n_{\bar{n}}(f_0)\ar[r] \ar[d] & Q^n_{\bar{n}}(f_1f_0) \ar[r] \ar[d] &
	Q^n_{\bar{n}}(f_2f_1f_0) \ar[r] \ar[d] & \dotso \ar[r] &
  Q^n_{\bar{n}}(\bar{f}_{\kappa}) \ar[d] \\
  Q^n_{\bar{n}+1}(f_0)\ar[r] & Q^n_{\bar{n}+1}(f_1f_0) \ar[r] &
	Q^n_{\bar{n}+1}(f_2f_1f_0) \ar[r] & \dotso \ar[r] &
	Q^n_{\bar{n}+1}(\bar{f}_{\kappa}),
} 
\]
thereby showing $Q^n_{\bar{n}}(\bar{f}_{\kappa})\to Q^n_{\bar{n}+1}(\bar{f}_{\kappa})$ is a $S$ $\Sigma$-inj $\Sigma_{n}$-proj cofibration. Since $Q^n_{0}(\bar{f}_{\kappa})=A^{\wedge n}$, this finishes the proof.
\end{proof}





\section{Cofibrancy of operadic constructions}\label{APPLICATIONS SEC}

The goal of this section is to prove Theorems \ref{CIRCO POS THM}, \ref{MODMODELEXIST THM}, \ref{FORGETFUL THM}, \ref{BARCONSTRUCTION THM} and \ref{FIBERSEQ THM}.

Subsection \ref{OPDEFINITIONS SEC} recalls some required operadic terminology and basic results. 

Subsection \ref{FILTRATIONS SEC} proves Proposition \ref{FILT PROP}, a filtration result that is key to the proof of Theorem \ref{CIRCO POS THM}. 

Subsection \ref{POSSYM SEC} extends the model structures of Section \ref{SINJPROJMOD SEC} to the category $\Sym$ of spectral symmetric sequences and proves for them analogues of the key results in Section \ref{PROPERTIES SEC}. 

Finally, the main proofs are found in subsections \ref{MAINPROOF SEC} and \ref{PROOFS SEC}.

\subsection{Definitions: operads, modules and algebras}\label{OPDEFINITIONS SEC}

We now recall some standard operadic terminology. We do so in terms of a general closed symmetric monoidal category $\C$ in order to greatly streamline the proof of Theorem \ref{CIRCO POS THM}. Indeed, even when proving only the algebra case of Theorem \ref{CIRCO POS THM}, Definition \ref{MA DEF} makes it necessary to nonetheless understand left modules, making it convenient to unify the discussion using Proposition \ref{MODRED PROP}.

\begin{definition}
Let $(\C,\otimes,\mathds{1})$ denote a closed symmetric monoidal category. 

The category $\Sym(\C)$ of {\it symmetric sequences in \C} is the category of functors $\Sigma \to \C$.

Further, for $G$ a finite group the category $\Sym^G(\C)$ of {\it $G$-symmetric sequences in \C} is the category of functors $G \to \Sym(\C)$.
\end{definition}

\begin{remark}
A symmetric sequence $X$ is formed by objects $X(r)\in \C, r\geq 0$ each with a left $\Sigma_r$-action. To avoid confusion when $\C=\mathsf{Sp}^\Sigma$ or a related category, we reserve the letter $r$ for this external index and keep $m$ as the internal spectrum index, so that $X_m(r)$ denotes the $m$-th simplicial set of $X(r)$.
\end{remark}

We now recall the two usual monoidal structures on $\Sym(\C)$. For our purposes the composition product $\circ$ is the most important of the two, with the tensor product $\check{\otimes}$ playing an auxiliary role.

\begin{definition}\label{TWOMON DEF}
Given $X,Y\in \Sym(\C)$ we define their {\it tensor product} to be
$$ (X \check{\otimes} Y)(r) = 
 \coprod_{0 \leq \bar{r} \leq r} \Sigma_{r} 
 \underset{\Sigma_{\bar{r}} \times \Sigma_{r-\bar{r}}}{\cdot}
 X(\bar{r}) \otimes Y(r-\bar{r}) $$ 
and their {\it composition product} to be
\begin{equation} \label{CIRC DEF} 
(X \circ Y)(r) = \coprod_{\bar{r} \geq 0} X(\bar{r}) \otimes_{\Sigma_{\bar{r}}} \left(Y^{\check{\otimes} \bar{r}}(r)\right).
\end{equation} 
\end{definition}

One has the following result (for a discussion of reflexive coequalizers see for example \cite[Def. 3.26]{Ha09} and the propositions immediately following it).

\begin{proposition}\label{SYMMON PROP}
Let $(\C,\otimes,\mathds{1})$ be a closed symmetric monoidal category with initial object $\emptyset$. Then
\begin{itemize}
\item $(\Sym, \check{\otimes},\check{\mathds{1}})$ is a closed symmetric monoidal category, with unit
$\check{\mathds{1}}(0)=\mathds{1}$, $\check{\mathds{1}}(r)=\emptyset,r\geq 1$;
\item  $(\Sym, \circ, \mathcal{I})$ is a (non-symmetric) monoidal category, with unit
$\mathcal{I}(1)=\mathds{1}$, $\quad \mathcal{I}(r)=\emptyset, r\neq 1$.

Further, $\circ$ commutes with all colimits in the first variable and with filtered colimits and reflexive coequalizers in the second variable.
\end{itemize}
\end{proposition}

\begin{definition}
An \textit{operad} $\O$ in $\C$ is a monoid object in $\Sym(\C)$ with respect to $\circ$, i.e., a symmetric sequence $\O$ together with multiplication and unit maps
$$\O \circ \O \to \O,\qquad \mathcal{I} \to \O$$
satisfying the usual associativity and unit conditions.
\end{definition}

\begin{definition}\label{OPERMOD DEF}
Let $\O$ be an operad in $\C$. A {\it left module} $N$ (resp. {\it right module} $M$) over $\O$ is an object in $\Sym(\C)$ together with a map
$$\O \circ N \to N \qquad (\text{resp. } M \circ \O \to M)$$
satisfying the usual associativity and unit conditions. The category of left modules (resp. right modules) over $\O$ is denoted $\mathsf{Mod}^l_{\O}$ (resp. $\mathsf{Mod}^r_{\O}$).
Further, left modules $X$ over $\O$ concentrated in degree $0$ (i.e. such that $X(r)=\emptyset$ for $r\geq 1$) are called  {\it algebras} over $\O$. The category of algebras over $\O$ is denoted $\mathsf{Alg}_{\O}$.
\end{definition}

\begin{proposition}
The categories $\mathsf{Mod}^r_{\O}$, $\mathsf{Mod}^l_{\O}$ and $\mathsf{Alg}_{\O}$ have all small limits and colimits.

Further, all limits and colimits in $\mathsf{Mod}^r_{\O}$ are underlying in $\Sym(\mathcal{C})$, and likewise for all limits, filtered colimits and reflexive coequalizers in both $\mathsf{Mod}^l_{\O}$ and $\mathsf{Alg}_{\O}$.
\end{proposition}

\begin{definition}
Given $M \in \mathsf{Mod}^r_{\O}$, $N \in \mathsf{Mod}^l_{\O}$, their \textit{relative composition product} is the reflexive coequalizer
\[M \circ_{\O} N = \colim (M \circ \O \circ N \rightrightarrows M \circ N).\] 
\end{definition}

\begin{lemma}\label{CIRCCOLIM LEM}
Consider the bifunctors  
$$\mathsf{Mod}^l_{\O} \times \Sym(\C) \xrightarrow{\minus \circ \minus} \mathsf{Mod}^l_{\O}, \qquad
  \mathsf{Mod}^r_{\O} \times \mathsf{Mod}^l_{\O} \xrightarrow{\minus \circ_{\O} \minus} \Sym(\C).$$
$\circ$ preserves any colimit in the $\mathsf{Mod}^l_{\O}$ variable and $\circ_{\O}$ preserves reflexive coequalizers and filtered colimits in the $\mathsf{Mod}^l_{\O}$ variable.
\end{lemma}

\begin{proof}
Since any $M \in \mathsf{Mod}^l_{\O}$ is a reflexive coequalizer $\colim(\O \circ \O \circ M \rightrightarrows \O \circ M)$ of free left modules, it suffices to verify the claim for diagrams of free left modules and free maps, and for those the result follows by Proposition \ref{SYMMON PROP}.
\end{proof}

\begin{remark}
We will also use the analogue of Definition \ref{OPERMOD DEF} for the category $\Sym^G(\C)$. 
One has a formal analogue of Proposition \ref{SYMMON PROP} for $\Sym^G(\C)$ using the same monoidal structures $\check{\otimes}$ and $\circ$ (with diagonal $G$-action) and units (with trivial $G$-action), so that operads and their left modules, right modules and algebras in $\Sym^G(\C)$ are defined just as above.
\end{remark}

Iterating the $\Sym$ construction will allow us to use Proposition \ref{MODRED PROP} to reduce the study of left modules to that of algebras.

\begin{definition}
The category $\BSym(\C)$ of \textit{bi-symmetric sequences in $\C$} is the category $\Sym(\Sym(\C))$ of symmetric sequences of symmetric sequences in $\C$.
\end{definition}

\begin{remark}
Since an object $X\in \BSym(\C)$ is formed by objects $X(r,s)\in \C$, $r,s \geq 0$ with  $\Sigma_r \times \Sigma_s$-actions one has two different inclusions
\[(\minus)^{\mathsf{r}} \colon \Sym(\C) \hookrightarrow \BSym(\C),\qquad 
  (\minus)^{\mathsf{s}} \colon \Sym(\C) \hookrightarrow \BSym(\C)\]
defined by
\[X^{\mathsf{r}}(r,s)= \begin{cases} X(r), & \text{ if } s=0 \\ \emptyset, & \text{ if } s\neq 0 \end{cases} ,\qquad X^{\mathsf{s}}(r,s)= \begin{cases} X(s), & \text{ if } r=0 \\ \emptyset, & \text{ if } r\neq 0 \end{cases}.\]
\end{remark}

Following Definition \ref{TWOMON DEF} one can build two monoidal structures in $\BSym(\C)$ which we denote by $\check{\check{\otimes}}$ and $\check{\circ}^{\mathsf{r}}$. Here we mark the composition product $\check{\circ}^{\mathsf{r}}$ to indicate that $r$ is kept as the operadic index. Note that while $\check{\check{\otimes}}$ behaves symmetrically with respect to the indexes $r$ and $s$, $\check{\circ}^{\mathsf{r}}$ does not.

\vskip 5pt

Both of the following results follow by a straightforward calculation.

\begin{proposition}
$(\minus)^{\mathsf{r}}, (\minus)^{\mathsf{s}}$ are monoidal functors from the symmetric mo\-noidal structure $\check{\otimes}$
to the symmetric monoidal structure $\check{\check{\otimes}}$.

$(\minus)^{\mathsf{r}}$ is a monoidal functor from the monoidal structure $\circ$ to the monoidal structure $\check{\circ}^{\mathsf{r}}$.
\end{proposition}

\begin{proposition}\label{MODRED PROP}
Let $\O$ be an operad in $\C$. There is a natural isomorphism of categories
\[(\minus)^{\mathsf{s}}\colon \mathsf{Mod}^l_{\O}(\C) \xrightarrow{\simeq} \mathsf{Alg}_{\O^{\mathsf{r}}}(\Sym(\C)).\]
\end{proposition}

\subsection{Filtrations}\label{FILTRATIONS SEC}

This subsection proves Proposition \ref{FILT PROP}, which provides the key filtrations to prove Theorem \ref{CIRCO POS THM}. These filtrations are adapted from \cite{ElmMan06}, \cite{Ha09}, among others, except we here show that such filtrations also hold after applying $M \circ_{\O} (\minus)$ for some $M \in \mathsf{Mod}^r_{\O}$. This is partly enabled by an alternate definition of $\O_A$. 

\begin{definition}\label{MA DEF}
Let $\O$ be an operad in $\C$ and $A\in \mathsf{Alg}_{\O}$ regarded as an element of $\mathsf{Mod}_{\O}^l$. 
We define
\begin{equation}\label{OA DEF EQ}
\O_A=\O \amalg A,
\end{equation}
were the coproduct is taken in $\mathsf{Mod}_{\O}^l$.
Additionally, for $M\in \mathsf{Mod}_{\O}^r$ we define
\[M_A=M\circ_{\O} \O_A.\]
\end{definition} 

\begin{remark}
As noted to the author by Harper, (\ref{OA DEF EQ}) appeared as \cite[Prop. 5.52]{HaHe13}. However, we benefit from using 
(\ref{OA DEF EQ}) as our \textit{definition} of $\O_A$, hence streamlining the proofs of Propositions \ref{MA PROP}, \ref{MOX PROP} versus similar results in \cite{HaHe13}.
\end{remark}

\begin{remark}\label{0ADJ RMK}
There are adjunctions
$$\iota \colon \mathsf{Alg}_{\O}
   \rightleftarrows 
	\mathsf{Mod}_{\O}^l\colon (\minus)(0) \qquad 
	(\minus)(0)\colon \mathsf{Mod}_{\O}^l\rightleftarrows \mathsf{Alg}_{\O}\colon \widetilde{(\minus)}$$
where $\iota$ is the inclusion and $\tilde{A}(0)=A$, $\tilde{A}(r)=\emptyset$ for $r\geq 1$. 
In particular, colimits in $\mathsf{Alg}_{\O}$ can be computed after the inclusion into $\mathsf{Mod}^l_{\O}$ and $\O_A(0)=A$.
\end{remark}

\begin{proposition}\label{MA PROP}
Let $A\in \mathsf{Alg}_{\O}$ and $X\in \Sym(\C)$. Then there is a natural isomorphism of $\mathsf{Mod}^l_{\O}$-valued functors 
\[(\O \circ X) \amalg A \simeq \O_A\circ X = \coprod_{r \geq 0} \O_A(r) \otimes_{\Sigma_r} X^{\check{\otimes} r}.\]
Additionally, for $M \in \mathsf{Mod}_{\O}^r$ there is a natural isomorphism of $\Sym(\C)$-valued functors
\[ M \circ_{\O} ((\O\circ X) \amalg A) \simeq M_A\circ X = \coprod_{r \geq 0} M_A(r) \otimes_{\Sigma_r} X^{\check{\otimes} r}.\]
\end{proposition}

\begin{proof}
We compute (applying Lemma \ref{CIRCCOLIM LEM} to the coproduct $\O \amalg A$)
$$\O_A\circ X=(\O \amalg A) \circ X \simeq
  (\O \circ X) \amalg (A \circ X) = (\O \circ X) \amalg A,$$
where $A \circ X=A$ since $A$ is in degree $0$. The additional claim is obvious.
\end{proof}

\begin{proposition}\label{MOX PROP}
Given $M \in \mathsf{Mod}^r_{\O}$, $X\in \C$ and 
$A \in \mathsf{Alg}_{\O}$ one has natural isomorphisms of $\Sym(\C)$-valued functors
\begin{equation}\label{MAX EQ} 
  M_{\O X \amalg A}(\minus) =
	\left(M\circ_{\O}(\O \amalg \O \circ X \amalg A)\right)(\minus) \simeq 
	\coprod_{r\geq 0}M_A(r+(\minus))\otimes_{\Sigma_{r}}X^{\otimes r}.
\end{equation}
\end{proposition}

\begin{proof}
This follows formally using the $(\minus)^{\mathsf{r}},(\minus)^{\mathsf{s}}$ functors. Combining Proposition \ref{MODRED PROP} to change perspective to 
$\mathsf{Alg}_{\O^{\mathsf{r}}}(\Sym(\C))$ with Proposition \ref{MA PROP} yields
\begin{equation}\label{BLA EQ}
  \left(M\circ_{\O}(\O \amalg \O \circ X \amalg A)\right)^{\mathsf{s}} \simeq
  M^{\mathsf{r}}\circ^{\mathsf{r}}_{\O^{\mathsf{r}}}
	(\O^{\mathsf{s}}\amalg \O^{\mathsf{r}}\circ^{\mathsf{r}}X^{\mathsf{s}} \amalg A^{\mathsf{s}})\simeq
	M^{\mathsf{r}}_{\O^{\mathsf{s}} \amalg A^{\mathsf{s}}}\circ^{\mathsf{r}} X^{\mathsf{s}}. 
\end{equation}
Applying Proposition \ref{MA PROP} and noting $A^{\mathsf{s}}=A^{\mathsf{r}}$ (as $A$ is an algebra) we compute 
\begin{align*} 
  M^{\mathsf{r}}_{\O^\mathsf{s} \amalg A^{\mathsf{s}}} = &
	M^{\mathsf{r}} \circ^{\mathsf{r}}_{\O^{\mathsf{r}}} \left(\O^{\mathsf{r}} \amalg \O^{\mathsf{s}} \amalg A^{\mathsf{s}}\right) =
	M^{\mathsf{r}}\circ^{\mathsf{r}}_{\O^{\mathsf{r}}} 
	\left({\O^{\mathsf{r}}}\circ^{\mathsf{r}}(\mathcal{I}^{\mathsf{r}} \amalg \mathcal{I}^{\mathsf{s}}) 
\amalg A^{\mathsf{s}} \right) \simeq \\
 \simeq &
  M_{A^{\mathsf{s}}}^{\mathsf{r}}\circ^{\mathsf{r}}(\mathcal{I}^{\mathsf{r}}\amalg \mathcal{I}^{\mathsf{s}}) = 
	M_{A^{\mathsf{r}}}^{\mathsf{r}}\circ^{\mathsf{r}}(\mathcal{I}^{\mathsf{r}}\amalg \mathcal{I}^{\mathsf{s}}) \simeq 
	\left(M_{A}\right)^{\mathsf{r}}\circ^{\mathsf{r}}(\mathcal{I}^{\mathsf{r}}\amalg \mathcal{I}^{\mathsf{s}}),
\end{align*} 
showing $M^{\mathsf{r}}_{\O^{\mathsf{s}} \amalg A^{\mathsf{s}}}(r,s) \simeq M_A(r+s)$. Plugging into (\ref{BLA EQ}) finishes the proof.
\end{proof}

We now turn to the key result in the subsection. $Q^r_{r-1}(f)$ is defined in Remark \ref{QNT RMK}.

\begin{proposition}\label{FILT PROP}
Consider any pushout in $\mathsf{Alg}_{\O}(\C)$ of the form
\begin{equation}\label{PUSHOUT EQ}
\xymatrix@1{
\O \circ X \ar[d]_{\O \circ f}\ar[r]^-h& A \ar[d]\\
\O \circ Y \ar[r]                 & B,         } \end{equation}
and let $M\in \mathsf{Mod}^r_{\O}$. Then, in the underlying category $\C$,
\begin{equation}\label{BFILT EQ}
M\circ_{\O}B \simeq \colim \left(A^{M}_0\to A^M_1\to A^M_2\to \cdots \right)
\end{equation}
where $A^M_0=M\circ_{\O}A$ and the $A^M_r$ are built inductively from pushout diagrams
\begin{equation} \label{IND EQ}
\xymatrix@1{
M_A(r)\otimes_{\Sigma_r} Q^{r}_{r-1}(f) \ar[d] \ar[r]& A_{r-1}^M \ar[d]\\
M_A(r)\otimes_{\Sigma_r} Y^{\otimes r}   \ar[r]       & A^M_r.} 
\end{equation}
\end{proposition}

\begin{remark} \label{MODRED RMK}
To streamline the proof of Theorem \ref{CIRCO POS THM} we will need to apply Proposition \ref{FILT PROP} to the category 
$\mathsf{Mod}^l_{\O}$, a move enabled by Proposition \ref{MODRED PROP}. 
This is mostly straightforward, with occurrences of $\C$ replaced by $\Sym(\C)$ and $\otimes$ replaced by $\check{\otimes}$, though defining $M_N \in \BSym(\C)$ when $N \in \mathsf{Mod}_{\O}^l$ requires some care. Analyzing Proposition \ref{MODRED PROP} and Definition \ref{MA DEF} leads to the definition 
\begin{equation}\label{MN DEF EQ}
M_{N} = M^{\mathsf{r}}_{N^{\mathsf{s}}}=
  M^{\mathsf{r}} \phantom{|}\check{\circ}^{\mathsf{r}}_{\O^{\mathsf{r}}} (\O^{\mathsf{r}}\amalg N^{\mathsf{s}}).
\end{equation}
Note that this is compatible with Definition \ref{MA DEF} when $N=A$ is an algebra since then $A^{\mathsf{s}}=A^{\mathsf{r}}$ so that 
$M^{\mathsf{r}}_{A^{\mathsf{s}}}=
M^{\mathsf{r}}_{A^{\mathsf{r}}}=
(M_A)^{\mathsf{r}}$.
\end{remark}

\begin{remark}
When $M=\O$, (\ref{MN DEF EQ}) appeared as \cite[Prop. 5.52]{HaHe13}, albeit with transposed indexes due to differing conventions.
Our convention has one nice advantage: filtrations of $M_A$ follow by Remark \ref{MODRED RMK} without a separate proof.
\end{remark}

The remainder of the subsection is dedicated to proving Proposition \ref{FILT PROP}. We essentially adapt the proof of \cite[Prop. 4.20]{Ha09}, although we substantially repackage the argument using a more categorical perspective. 

For motivation we note that, in short, the proof of \cite[Prop. 4.20]{Ha09} follows by noticing that 
$A\amalg_{\O \circ X}\O \circ Y$ 
is built out of terms of the form
\begin{equation}\label{WORDY EQ}
\O_A\left(\left| w \right|\right)\otimes w^{\otimes}(X,Y),
\end{equation}
where $w^{\otimes}(X,Y)$ denotes a \textit{word} (or \textit{non abelian monomial}) for the operation $\otimes$ in the \textit{letters} $X,Y$ (e.g. $X \otimes Y \otimes X$, $Y \otimes Y \otimes X \otimes Y$) and $\left| w \right|$ its length, glued
along certain maps between them\footnote{We note, however, that the need to deal with \textit{non abelian monomials}, rather than abelian ones, is somewhat hidden in the presentation of the proof in \cite{Ha09}. We recommend the reader interested in gleaning why these are needed to focus on the last two paragraphs of that proof. 
}. The filtration (\ref{BFILT EQ}) is then obtained by analyzing a long list of compatibility relations satisfied by those maps. 

One drawback of such an approach is that it can be hard to keep track of the compatibilities that need to be verified. Instead, our approach will be to first identify a ``diagram category of words'' $\mathcal{W}$ together with a functor $F \colon \mathcal{W} \to \mathcal{C}$ defined on objects by (\ref{WORDY EQ}) and for which $A\amalg_{\O \circ X}\O \circ Y = \colim_{\mathcal{W}}F$. 
Then, since $\mathcal{W}$ will encode all the necessary maps and compatibilities, the desired filtration (\ref{BFILT EQ}) will follow from a filtration $\mathcal{W}_{\leq r}$ of $\mathcal{W}$ itself.

\vskip 5pt

To motivate the definition, note that $\mathcal{W}$ needs enough arrows to describe:
\begin{enumerate*}
\item[(i)] the $Q^r_{r-1}$ constructions of Remark \ref{QNT RMK}, i.e., $\mathcal{W}$ should contain the ``$r$-cube categories'' $(x \to y)^{\times r}$; 
\item[(ii)] the $\Sigma_r$-action on $Q^r_{r-1}$;
\item[(iii)] maps between the terms in (\ref{WORDY EQ}) that remove some $X$ letters (induced by $h \colon \O \circ X \to A$ in (\ref{PUSHOUT EQ})). 
\end{enumerate*}

This desiderata will likely remind some readers of Grothen\-dieck constructions (cf. for example \cite[Construction 7.1.9]{Ri14}). 

\begin{definition}
Consider the functor ($\mathsf{Fininj}$ denoting (a skeleton of) finite sets and injections)
$ (x \to y)^{(\minus)} \colon \mathsf{Fininj}^{op} \to \mathsf{Cat} $
 defined by $\underline{r} =\{1,2,\cdots,r \} \mapsto (x \to y)^{ \underline{r}}, r \geq 0$ and  let $\mathcal{G}$ denote the corresponding Grothendieck construction.  
Explicitly, objects of $\mathcal{G}$ are pairs 
$$(\underline{r},w \in (x \to y)^{\underline{r}})$$
and an arrow $(\underline{r},w) \to (\underline{r_{\**}},w_{\**})$ is a pair 
$$(\iota \colon \underline{r_{\**}}\hookrightarrow \underline{r},w \circ \iota \to w_{\**})$$
(note that, since $(x \to y)^{\underline{r_{\**}}}$ is a poset, arrows are determined by their first component) with composition defined in the obvious way.
\end{definition}

\begin{notation}
To ease notation we will when convenient refer to an object of $\mathcal{G}$ by its second component $w$ and to an arrow by its first component $\iota$.
\end{notation}

\begin{remark}\label{GDESC REM}
Objects $(\underline{r},w) \in \mathcal{G}$ can be identified with words $w(x,y)$ on letters $x$ and $y$ where $r=|w|$, the length of the word. Further, we let $|w|_x$ (resp. $|w|_y$) denote the number of $x$'s (resp. $y$'s) in $w$.

Given a word $w=l_1 l_2\cdots l_r$ with $l_i \in \{x,y\}$ and an injection 
$\iota \colon \underline{r_{\**}} \hookrightarrow \underline{r}$, one has
$w \circ \iota= l_{\iota(1)} l_{\iota(2)}\cdots l_{\iota(r_{\**})}$, which we think of as the word obtained by removing the letters of $w$ in positions not in $\iota(r_{\**})$ and suitably shuffling the remaining letters.
An arrow $w \to w_{\**}$ can then be interpreted as an injection 
$\iota \colon \underline{\left|w_{\**}\right|} \hookrightarrow \underline{\left|w\right|}$ such that, after removing and shuffling letters of $w$ to obtain $w \circ \iota$, one can obtain $w_{\**}$ from $w \circ \iota$ by further replacing some $x$'s with $y$'s (now without shuffling).

Note that hence any $\iota \colon w \to w_{\**}$ has a natural factorization $w \to w \circ \iota \to w_{\**}$.
\end{remark}

Comparing the description in Remark \ref{GDESC REM} with the desiderata for $\mathcal{W}$, we see that $\mathcal{G}$ has more arrows than desired, namely those that remove $y$ letters. 

\begin{definition}
$\mathcal{W}$ is the subcategory of $\mathcal{G}$ with all objects and those arrows 
$\iota \colon (\underline{r},w) \to (\underline{r_{\**}},w_{\**})$ such that 
$w\left(\underline{r}-\iota\left(\underline{r_{\**}}\right)\right)
  \subset \{x\}$ 
or, equivalently,
$|w|_y=|w \circ \iota|_y$.
	
Further, for each $r \geq 0$, let $\mathcal{W}_{\leq r}$ (resp. $\mathcal{W}_{r}$) denote the full subcategory of those $w \in \mathcal{W}$ satisfying $|w|\leq r$ (resp. $|w|=r$).
\end{definition}

\begin{notation}
It will be convenient to name certain types of arrows in $\mathcal{W}$:
\begin{itemize}
\item a \textit{shuffle} is an arrow $\sigma \colon w \to w \circ \sigma$ for $\sigma \in \Sigma_{|w|}^{op}$;
\item a \textit{tidy arrow} is an arrow $\pi \colon w=\bar{w}x^a \to w_{\**}$ for 
$\pi$ the inclusion  $ \underline{|w_{\**}|} = \underline{|\bar{w}|} \subset \underline{|w|}$;
\item a \textit{removing arrow} is an arrow $w \to y^{|w|_y}$;
\item a \textit{replacing arrow} is an arrow $w \to y^{|w|}$.
\end{itemize}
\end{notation}

\begin{definition}
$\bar{\mathcal{W}}$ is the subcategory of $\mathcal{W}$ with the same objects but only the shuffles, removing and replacing arrows. 
\end{definition}

\begin{remark}
Keeping the intuition of Remark \ref{GDESC REM}, tidy arrows remove some $x$'s at the end of a word and then replace some $x$'s by $y$'s without any shuffling, removing arrows remove all $x$'s (perhaps shuffling) and replacing arrows replace all $x$'s by $y$'s (perhaps shuffling).
\end{remark}

The key to proving Proposition \ref{FILT PROP} are the following lemmas building $F \colon \mathcal{W} \to \C$ and establishing some categorical results about $\mathcal{W}$.

\begin{lemma}
The pushout diagram (\ref{PUSHOUT EQ}) and $M \in \mathsf{Mod}^r_{\O}(\C)$ naturally induce a functor
\[F^M \colon \mathcal{W} \to \mathcal{C}.\]
\end{lemma}

\begin{proof}
We define $F^M$ on objects in the obvious way as 
\begin{equation}\label{FDEF EQ}
F^M(w(x,y))=M_A(\left|w\right|) \otimes w^{\otimes}(X,Y).
\end{equation}
For arrows, we first declare that for a shuffle $\sigma \colon w \to w \circ \sigma$, 
$$M_A(\left|w\right|) \otimes w^{\otimes}(X,Y) \xrightarrow{F^M(\sigma)} M_A(\left|w\right|) \otimes (w\circ \sigma)^{\otimes}(X,Y)$$
is the map defined by the action of $\sigma^{-1} \in \Sigma_{\left|w\right|}$ on $M_A(\left|w\right|)$ and by shuffling $w^{\otimes}(X,Y)$. Since any arrow can be made tidy by pre-composing with a shuffle, it remains to coherently define $F^M$ on tidy arrows. For a tidy arrow $\pi \colon w = \bar{w}(x,y)x^a \to w_{\**}(x,y)$, define $F^M(\pi)$ via the diagram (with vertical maps the summand inclusions induced by Proposition \ref{MA PROP} and writing $\O (\minus)$ for
$\O \circ (\minus)$) 
\begin{equation}\label{FTIDY EQ}
\begin{tikzcd}[column sep=5.3pc]
M_A(\left|w\right|) \otimes \bar{w}^{\otimes}(X,Y) \otimes X^{\otimes a} \ar{r}{F^M(\pi)} \arrow[hook]{d}   & 
M_A(\left|w_{\**}\right|) \otimes w_{\**}^{\otimes}(X,Y)                                               \arrow[hook]{d}
\\
M \circ_{\O} \left( \O ( \bar{w}^{\amalg}(X,Y)) \amalg \O X^{\amalg a} \amalg A \right) 
\arrow{r}{M \circ_{\O}\left(\O f_{\**} \amalg h_{\**} \amalg id_A\right)}  & 
M \circ_{\O} \left( \O (w_{\**}^{\amalg}(X,Y)) \amalg A \right). 
\end{tikzcd}
\end{equation}
$F^M(\pi)$ is well defined since $\pi \sigma$ is tidy only for $\sigma \in \Sigma_a \subset \Sigma_{|w|}$ and such shuffles do not change (\ref{FTIDY EQ}). It follows that $F^M$ is well defined on all arrows.

We now verify $F^M$ respects compositions. This is clear when composing either two shuffles or two tidy arrows, and since general two-fold compositions factor as 
$w \xrightarrow{\sigma} w \circ \sigma \xrightarrow{\pi} w_{\**} \xrightarrow{\sigma_{\**}} w_{\**} \circ \sigma_{\**} \xrightarrow{\pi_{\**}}w_{\**\**}$ 
with $\sigma,\sigma_{\**}$ shuffles and $\pi,\pi_{\**}$ tidy, it remains to show $F^M(\sigma_{\**} \pi)=F^M(\sigma_{\**}) F^M(\pi)$.
Identifying $\sigma_{\**} \in \Sigma_{|w_{\**}|}\subset \Sigma_{|w|}$, one has $\sigma_{\**} \pi \sigma_{\**}^{-1}$ tidy, so that by definition $F^M(\sigma_{\**}\pi)=F^M(\sigma_{\**}\pi\sigma_{\**}^{-1})F^M(\sigma_{\**})$. The claim now follows since (\ref{FTIDY EQ}) respects the action of $\sigma_{\**}$.
\end{proof}

Recall that a functor $\mathcal{J} \to \mathcal{I}$ between diagram categories is called \textit{final} if for any functor 
$F \colon \mathcal{I} \to \C$ (where $\C$ is a category with all small colimits) one has 
 \[\colim_{\mathcal{I}}F=\colim_{\mathcal{J}}F|_{\mathcal{J}}.\]
We will need several finality conditions for subcategories of $\mathcal{W}$. In all cases we show them by verifying (cf. \cite[IX.3 Thm. 1]{McL}) that for all $i \in \mathcal{I}$ the under categories $i \downarrow \mathcal{J}$ are non-empty (this will always be obvious) and connected.

\begin{lemma}\label{FINAL1 LEM}
The subcategory $\bar{\mathcal{W}}$ is final in $\mathcal{W}$. 
\end{lemma}

\begin{proof} $w \downarrow \bar{\mathcal{W}}$ is connected iff any two arrows 
in $\mathcal{W}$ with source $w$ are connected by a zigzag of post-compositions with arrows in $\bar{\mathcal{W}}$. 
For such an arrow $\iota \colon w \to w_{\**}$ the natural decomposition $w \to w \circ \iota \to w_{\**}$ satisfies $|w|_y=|w \circ \iota|_y$, $|w \circ \iota|=|w_{\**}|$, so that by picking any arrows $w_{\**}\to y^{|w_{\**}|}$ and $w \circ \iota \to y^{|w|_y}$ one has a diagram
\begin{equation}\label{SURPUSEFUL EQ}
\xymatrix@1{ 
            &                 &  w   \ar[d] \ar[rd]^-{\mathsf{rm}}                                   &           \\
	y^{|w_{\**}|}  & w_{\**} \ar[l]^-{\mathsf{rp}}  &  w \circ \iota \ar@/_1.5pc/[ll]_-{\mathsf{rp}}   \ar[r]_-{\mathsf{rm}} \ar[l]           & y^{|w|_y}
}\end{equation}
where arrows marked $\mathsf{rp}$ are replacing and arrows marked $\mathsf{rm}$ are removing. The marked arrows exhibit a zigzag in $w \downarrow \bar{\mathcal{W}}$ between $\iota$ and $id_w$.
\end{proof}

\begin{lemma}\label{WYR LEM}
Let $\bar{\mathcal{W}}_{y^r}$ denote the full subcategory of $\bar{\mathcal{W}}$ of objects that admit arrows to $y^r$. The group $\Sigma_{y^r}$ of shuffles of $y^r$ is final in $\bar{\mathcal{W}}_{y^r}$.
\end{lemma}

\begin{proof}
Since any $w$ is isomorphic up to shuffle to some $y^bx^a$, it suffices to check all $y^bx^a \downarrow \Sigma_{y^r}$ are connected, i.e., that all arrows $y^bx^a \to y^r$ in $\bar{\mathcal{W}}$ are connected by post-composing with a shuffle. Both cases $b=r$ and $a+b=r$ are clear.
\end{proof}

\begin{lemma}\label{FINAL2 LEM}
The subcategory $\mathcal{W}_{\leq (r-1)}$ is final in $\mathcal{W}_{\leq r}-y^r$. 
\end{lemma}

\begin{proof} When $|w|\leq r-1$ one has an initial object $id_w$ in $w \downarrow \mathcal{W}_{\leq (r-1)}$, which is hence connected. 
When $|w|=r$, $w \downarrow \mathcal{W}_{\leq (r-1)}$ is connected precisely if any two arrows 
$w \to w_{\**}$ with $|w_{\**}|\leq r-1$ 
 are connected by a zigzag of post-compositions with arrows in $\bar{\mathcal{W}}_{\leq (r-1)}$. 
For any such arrow $\iota \colon w \to w_{\**}$ the natural decomposition $w \to w \circ \iota \to w_{\**}$ satisfies $|w|_y=|w \circ \iota|_y \leq r-1$, $|w \circ \iota|=|w_{\**}| \leq r-1$ so that diagram (\ref{SURPUSEFUL EQ}) exhibits a zigzag between $\iota$ and an arrow $w \to y^{|w|_y}$. As in Lemma \ref{WYR LEM}, all of the latter arrows are connected by post-composing with a shuffle (in fact, it suffices to check this for $w=y^bx^a$).
\end{proof}

\begin{lemma}\label{NERVE LEM}
$\mathcal{W}_{\leq r}=(\mathcal{W}_{\leq r}-y^r) \cup \mathcal{W}_{r}$. In fact, $N(\mathcal{W}_{\leq r})=N(\mathcal{W}_{\leq r}-y^r)\cup N(\mathcal{W}_{r})$. 
\end{lemma}

\begin{proof}
Arrows in $\mathcal{W}$ never decrease $|\minus|_y$, hence any string of arrows in $\mathcal{W}_{\leq r}$ involving the object $y^r$ must in fact be a string of arrows in $\mathcal{W}_{r}$.
\end{proof}

\begin{lemma} 
$$M \circ_\O B \simeq \colim_{\mathcal{W}}F^M.$$
\end{lemma}

\begin{proof} Note first that by Lemma \ref{FINAL1 LEM} it suffices to show $M \circ_\O B \simeq \colim_{\bar{\mathcal{W}}}F^M.$

By general considerations one can describe $B$ as a reflexive coequalizer
\[B \simeq \colim \left(  (\O \circ (X \amalg Y) \amalg A) \rightrightarrows (\O \circ Y) \amalg A \right)\]
and hence by Proposition \ref{MA PROP} and Lemma \ref{CIRCCOLIM LEM}
\begin{equation}\label{MOB EQ}
M \circ_{\O} B \simeq
\colim (\xymatrix@1{M_A \circ (X\amalg Y) \ar@<.5ex>[r]^-{f_{\**}} \ar@<-.5ex>[r]_-{h_{\**}} & M_A \circ  Y}).
\end{equation}
Now note that
\[
M_A\circ (X\amalg Y)=\coprod_{i,j\geq 0} 
M_A(i+j)\otimes_{\Sigma_{i}\times \Sigma_{j}}X^{\otimes i} \otimes Y^{\otimes j},
\]
with the reflexive map in (\ref{MOB EQ}) naturally identifying $M_A\circ Y$ with the subobject formed by the $i=0$ summands. Since by naturality of Propositions \ref{MA PROP} and \ref{MOX PROP} the maps being equalized in (\ref{MOB EQ}) send summands to summands, repackaging universal properties allows one to rewrite
\begin{equation}\label{COLIMDA EQ}
M\circ_{\O}B=\colim_{\mathcal{M}} \bar{F}^M. 
\end{equation}
Here $\mathcal{M}$ is the diagram category whose objects we denote by monomials $x^i y^j$, $i,j \geq 0$ together with unique non identity arrows $x^i y^j \to y^{i+j}$, $x^i y^j \to y^j$ for $i \neq 0$ (note that non identity arrows can never be composed). $\bar{F}^M$ is defined on objects by 
\begin{equation}\label{FADEF EQ}
\bar{F}^M(x^i y^j) = M_A(i+j)\otimes_{\Sigma_{i}\times \Sigma_{j}}X^{\otimes i}\otimes Y^{\otimes j},
\end{equation}
is induced on arrows $x^i y^j\to y^{i+j}$ by the map $f_{\**}$ in (\ref{MOB EQ}) and on arrows $x^i y^j\to y^j$ by the map $h_{\**}$.

There is an obvious functor $\bar{\mathcal{W}} \to \mathcal{M}$ defined by $w \mapsto x^{|w|_x}y^{|w|_y}$ (arrows are mapped in the only possible way and functoriality is trivial since non identity arrows in $\mathcal{M}$ can not be composed). We claim $\bar{F}^M=\Lan_{\bar{\mathcal{W}} \to \mathcal{M}} F^M$. By \cite[X.3.1]{McL}
\begin{equation}\label{LKAN EQ}
(\Lan_{\bar{\mathcal{W}} \to \mathcal{M}}F^M)(x^i y^j)=
  \colim_{\bar{\mathcal{W}} \downarrow x^i y^j}F^M|_{\bar{\mathcal{W}} \downarrow x^i y^j}
\end{equation}
When $i \neq 0$, $\bar{\mathcal{W}} \downarrow x^i y^j$ is just the groupoid of words $w$ with $|w|_x=i, |w|_y=j$, while for $y^r$ it is the category $\bar{\mathcal{W}}_{y^r}$ of Lemma \ref{WYR LEM} containing the final group $\Sigma_{y^r}$. In either case, the formula (\ref{LKAN EQ}) computes the quotient of the terms in (\ref{FDEF EQ}) by the obvious shuffle groupoid action and hence coincides with $\bar{F}^M$ on objects. To see (\ref{LKAN EQ}) also coincides with $\bar{F}^M$ on arrows consider the commutative diagrams (with vertical maps induced by codiagonals and writing 
$\O (\minus)$ for $\O \circ (\minus)$)
\[\xymatrix@1{
\O (Y^{\amalg j}\amalg X^{\amalg i}) \amalg A \ar[r]^-{f_{\**}} \ar[d]_{\nabla_{\**}} &
\O Y^{\amalg (i+j)} \amalg A \ar[d]^{\nabla_{\**}} &
\O(Y^{\amalg j}\amalg X^{\amalg i}) \amalg A \ar[r]^-{h_{\**}} \ar[d]_{\nabla_{\**}} &
\O Y^{\amalg j} \amalg A \ar[d]^{\nabla_{\**}}
\\
\O(Y \amalg X) \amalg A \ar[r]^-{f_{\**}} & \O Y \amalg A &
\O(Y \amalg X) \amalg A \ar[r]^-{h_{\**}} & \O Y \amalg A.
}\] 
Since $F^M$ is defined using (shuffles) of the top maps, and $\bar{F}^M$ is defined using the bottom maps, we conclude (\ref{LKAN EQ}) indeed equals $\bar{F}^M$ on maps. Noting that left Kan extensions have the same colimit finishes the proof.
\end{proof}

\begin{proof}[Proof of Proposition \ref{FILT PROP}]
By the previous lemma $M \circ_\O B \simeq \colim_{\mathcal{W}}F^M$. We define
$A^M_r=\colim_{\mathcal{W}_{\leq r}}F^M$, so that (\ref{BFILT EQ}) is immediate since the $\mathcal{W}_{\leq r}$ filter $\mathcal{W}$. It is straightforward to check that Lemma \ref{NERVE LEM} implies one has pushout diagrams 
\[\xymatrix@1{
  \colim_{\mathcal{W}_{r}-y^r}F^M                \ar[r]  \ar[d] &
  \colim_{\mathcal{W}_{ \leq r}-y^r}F^M                  \ar[d] \\
  \colim_{\mathcal{W}_{r}}F^M                            \ar[r] &
  \colim_{\mathcal{W}_{\leq r}}F^M,
}\]
and it hence suffices to verify these diagrams have the form (\ref{IND EQ}). The two diagrams coincide on the bottom right corner by definition and on the top right corner by Lemma \ref{FINAL2 LEM}. The left hand maps of the two diagrams are seen to coincide by direct computation since the tidy arrow subcategory of $\mathcal{W}_{r}$ is precisely $(x \to y)^{\underline{r}}$ and it is easy to check that
$\colim_{\mathcal{W}_{r}}F^M=\left(\colim_{(x \to y)^{\underline{r}}}F^M\right)_{\Sigma_r}$ and similarly for $\mathcal{W}_{r}-y^r$.
\end{proof}

\subsection{Model structures on $\Sym$ and $\Sym^G$}\label{POSSYM SEC}

\begin{notation} 
In what follows we abbreviate $\Sym(\mathsf{Sp}^{\Sigma})$ simply as $\Sym$.
\end{notation}

We now introduce for $\Sym$ the analogues of the model structures in Section \ref{SINJPROJMOD SEC} and show that the main results in 
Section \ref{PROPERTIES SEC} formally imply their $\Sym$ analogues.

\begin{definition}\label{SSYM DEF}
The \textit{$S$ stable (resp. monomorphism stable) model structure} on $\Sym$ is obtained by combining the $S$ stable (resp. monomorphism stable) model structures on $(\mathsf{Sp}^{\Sigma})^{\Sigma_r}$  in all degrees  (cf. Section \ref{MONOS SEC}).
\end{definition}

\begin{definition}\label{POSSYM DEF}
The {\it positive $S$ stable model structure} on $\Sym$ is the model structure obtained by combining the positive $S$ stable model structure in $\mathsf{Sp}^{\Sigma}$ on degree $r=0$ with the $S$ stable model structures on $(\mathsf{Sp}^{\Sigma})^{\Sigma_r}$  in degrees $r\geq 1$ (cf. Sections \ref{STABLEPOSSTABLE SEC} and \ref{MONOS SEC}).
\end{definition}

\begin{remark}
To motivate the use of the word ``positive'' in the previous definition, recall that each $X\in \Sym$ is composed of pointed simplicial sets $X_m(r)$, making it a bi-graded object. Since $\check{\otimes}$ is additive in both gradings, one can think of $m+r$ as the total degree of $X_m(r)$.
\end{remark}

We will also want to have an analogue for $\Sym^G$ of the $\Sigma$-inj $G$-proj $S$ stable model structure on $(\mathsf{Sp}^{\Sigma})^G$.

\begin{definition} \label{SIGMASIGMASIGMA DEF}
The {\it $S$ $\Sigma \times \Sigma$-inj $G$-proj stable model structure} on $\Sym^G$ is the model structure obtained by combining the $S$ $\Sigma \times \Sigma_r$-inj $G$-proj stable model structures on all degrees $r \geq 0$ (cf. Remark \ref{HSIGMAINJGPROJ RMK}).
\end{definition}

What follows are formal analogues for $\Sym$ of Propositions \ref{PROPSMASH} and \ref{PROPINDUCE} and Theorems \ref{BIQUILLEN THM} and \ref{SIGMANPUSHPROD THM}. 

\begin{proposition}\label{PROPSMASH SYM} 
Suppose all categories are equipped with their respective $S$ $\Sigma \times \Sigma$-inj $G$-proj stable model structure. Then the bifunctor
\[\Sym^G\times \Sym^{\bar{G}} \xrightarrow{\minus \check{\wedge} \minus} \Sym^{G \times \bar{G}} \]
is a left Quillen bifunctor.
\end{proposition}

\begin{proof}
Existence of the right adjoints is formal. Now recall that
$$ (X\check{\otimes}Y)(r) = \bigvee_{0\leq \bar{r} \leq r}\Sigma_{r} \underset{\Sigma_{\bar{r}}\times \Sigma_{r-\bar{r}}}{\cdot} X(\bar{r})\wedge Y(r-\bar{r}).$$
By injectiveness of the model structures (cf. Remark \ref{HSIGMAINJGPROJ RMK}), we can ignore the symmetric group actions, so that $\check{\wedge}$ is a wedge of bifunctors for each of which Proposition \ref{PROPSMASH} applies.  
\end{proof}

\begin{proposition}\label{PROPINDUCE SYM}
Let $\bar{G}\subset G$ be finite groups, and suppose each category is equipped with the respective $S$ $\Sigma \times \Sigma$-inj $G$-proj stable model structure. Then both adjunctions 
\[\fgt\colon \Sym^G  \rightleftarrows 
  \Sym^{\bar{G}}\colon ((\minus)^{G \cdot S})^{\bar{G}} \quad \text{ and } \quad
 G\times_{\bar{G}} (\minus) \colon (\Sym^\Sigma)^{\bar{G}} \rightleftarrows
  (\Sym^\Sigma)^G \colon \fgt\]
are Quillen adjunctions.
\end{proposition}

\begin{proof}
This is obvious from Proposition \ref{PROPINDUCE} since we are dealing with injective model structures (cf. Remark \ref{HSIGMAINJGPROJ RMK}).
\end{proof}

\begin{proposition}\label{BIQUILLEN THM SYM}
Consider the bifunctor 
\[\Sym^G\times \Sym^G \xrightarrow{\minus \check{\wedge}_{G} \minus} \Sym,\] 
where the first copy of $\Sym^G$ is regarded as equipped with the $S$ $\Sigma \times \Sigma$-inj $G$-proj stable model structure.
Then $\check{\wedge}_{G}$ is a left Quillen bifunctor if either:
\begin{enumerate}
\item[(a)] Both the second $\Sym^G$ and the target $\Sym$ are equipped with the respective monomorphism stable model structures;

\item[(b)] Both the second $\Sym^G$ and the target $\Sym$ are equipped with the respective $S$ stable model structures.
\end{enumerate}
\end{proposition}

\begin{proof}
This follows immediately by combining the ``wedge of bifunctors'' argument from the proof of Proposition \ref{PROPSMASH SYM} with Theorem \ref{BIQUILLEN THM}.
\end{proof}

\begin{proposition}\label{SIGMANPUSHPROD THM SYM}
Let $\Sym$ be equipped with the positive $S$ stable model structure and $\Sym^{\Sigma_n}$ with the $S$ $\Sigma \times \Sigma$-inj $\Sigma_n$-proj stable model structure.

Then for $f\colon A\to B$ a cofibration in $\Sym$ its $n$-fold pushout product
\[f^{\square n} \colon Q_{n-1}^n(f) \to B^n\]
is a cofibration in $\Sym^{\Sigma_n}$, which is a weak equivalence when $f$ is.

Furthermore, if $A$ is cofibrant in $\Sym$ then $Q_{n-1}^n(f)$ (resp. $f^{\check{\wedge} n} \colon A^{\check{\wedge}n} \to B^{\check{\wedge}n}$) is cofibrant (resp. cofibration between cofibrant objects) in $\Sym^{\Sigma_n}$.
\end{proposition}

\begin{proof}
Note first that by injectiveness (cf. Remark \ref{HSIGMAINJGPROJ RMK}) we need only worry about the $\Sigma_n$-actions and can ignore the $\Sigma_r$-actions.

Computing $X_1{\check{\wedge}\cdots \check{\wedge}}X_n$ iteratively and regrouping terms we get
$$(X_1{\check{\wedge}\cdots \check{\wedge}}X_n)(r)=
\bigvee_{\{\phi \colon \underline{r} \to \underline{n} \}} X_1(\phi^{-1}(1))\wedge\cdots \wedge X_n(\phi^{-1}(n)).$$
Since the shuffle isomorphisms for $\check{\wedge}$ involve a post-composition $\Sigma_n$-action on the set $\{\phi \colon \underline{r} \to \underline{n}\}$ indexing the wedge summands, the $\Sigma_n$-coset decomposition 
\begin{equation}\label{SIGMANDEC EQ}
(X_1{\check{\wedge}\cdots \check{\wedge}}X_n)(r)=
\bigvee_{\left(\bar{\phi}\right)\in \{\phi \colon \underline{r} \to 
\underline{n}\}/\Sigma_n}\bigvee_{\phi \in \left(\bar{\phi}\right) } X_1(\phi^{-1}(1))\wedge\cdots \wedge X_n(\phi^{-1}(n))
\end{equation}
is compatible with those shuffle isomorphisms, so that it suffices to verify the conclusions of the theorem for each of the subfunctors formed by the wedge summands over a single coset  $\left(\bar{\phi}\right)\in \{\phi \colon \underline{r} \to \underline{n}\}/\Sigma_n$. 

Now consider a map $f\colon A \to B$ in $\Sym$. Without loss of generality we can assume that the representative $\bar{\phi}$ misses precisely the first $\bar{n}$ elements in $n$, so that when computing $f^{\square n}$ the $\Sigma_n$-isotropy of the $\bar{\phi}$ wedge summand (i.e. the subgroup sending that summand to itself) is $\Sigma_{\bar{n}}$, and hence the component of $f^{\square n}$ corresponding to the $\left( \bar{\phi} \right)$ subfunctor in (\ref{SIGMANDEC EQ}) can be rewritten as
$$ \Sigma_n\underset{\Sigma_{\bar{n}}}{\cdot} f(0)^{\square^{\wedge} \bar{n}} \square^{\wedge}
 f\left(\phi^{-1}(\bar{n}+1)\right) \square^{\wedge} \cdots \square^{\wedge} f\left(\phi^{-1}(n)\right).$$
We need to show that this is a $S$ $\Sigma$-inj $\Sigma_n$-proj cofibration if $f$ is a positive $S$ cofibration. This follows by first applying Theorem \ref{SIGMANPUSHPROD THM} to $f(0)^{\square^{\wedge} \bar{n}}$, then applying Proposition \ref{PROPSMASH} to conclude 
$f(0)^{\square^{\wedge} \bar{n}} \square^{\wedge} f\left(\phi^{-1}(\bar{n}+1)\right) \square^{\wedge} \cdots \square^{\wedge} f\left(\phi^{-1}(n)\right)$ is a $S$ $\Sigma$-inj $\Sigma_{\bar{n}}$-proj cofibration, and finishing by applying Proposition \ref{PROPINDUCE}. 

The additional claims assuming $A$ is positive $S$ cofibrant follow by the same argument by noting that $f_1 \square^{\wedge} f_2, f_1 \wedge f_2$ are cofibrations between cofibrant objects if so are $f_1,f_2$ and using the additional statements in Theorem \ref{SIGMANPUSHPROD THM}.
\end{proof}

\begin{remark}\label{BSYM RMK}
All definitions and results in this subsection generalize to the category $\BSym=\Sym(\Sym)$. 
Indeed, one can define monomorphism, $S$ and positive $S$ stable model structures on $\BSym$ and $S$ $\Sigma\times \Sigma \times \Sigma$-inj $\Sigma_n$-proj stable model structures on $\BSym^{\Sigma_n}$ by just repeating Definitions \ref{SSYM DEF}, \ref{POSSYM DEF} and \ref{SIGMASIGMASIGMA DEF} except now replacing the initial structures on $\mathsf{Sp}^{\Sigma}$ with their eponymous analogues on $\Sym$. Further, analyzing the proofs of Propositions \ref{PROPSMASH SYM}, \ref{PROPINDUCE SYM}, \ref{BIQUILLEN THM SYM} and \ref{SIGMANPUSHPROD THM SYM} it is clear that those results themselves imply the analogue $\BSym$ results.
\end{remark}

\subsection{Proof of Theorem \ref{CIRCO POS THM}}\label{MAINPROOF SEC}

\begin{proof}[Proof of Theorem \ref{CIRCO POS THM}]
To simplify the discussion and notation somewhat, we first deal with the case where $f_2$ is a map in $\mathsf{Alg}_{\O}\subset \mathsf{Mod}^l_{\O}$. 

Writing $f_2 \colon A \to B$, note first that if $A=\O(0)$, then $f_1 \circ_{\O} A$ is a $S$ cofibration (resp. monomorphism), since 
$f_1 \circ_{\O} \O(0) \simeq f_1 \circ_{\O} \O \circ \** \simeq f_1 \circ * \simeq f_1(0)$. 
Otherwise, the same conclusion follows by first running the full proof for the map $\O(0) \to A$.

We now write $f_2$ as a retract of a transfinite composition of a $\kappa$-diagram $\mathsf{F} \colon \kappa \to \mathsf{Sp}^{\Sigma}$ where each successor map is a pushout of a generating cofibration, just as in (\ref{TRANSFINITE EQ}). As usual, retracts cause no difficulty so we reduce to the case of $f_2$ the transfinite composition of $\mathsf{F}$. Recalling that $\circ_{\O}$ commutes with transfinite compositions in the second variable (cf. Lemma \ref{CIRCCOLIM LEM}) and setting $f_1 \colon M \to N$, one sees that $f_1 \square^{\circ_{\O}} f_2$ will be a suitable cofibration provided that 
$M \circ_{\O} \mathsf{F} \to N \circ_{\O} \mathsf{F}$ is a $\kappa$-projective cofibration between $\kappa$-diagrams. 
Since cofibrancy at $\beta=0$ is satisfied due to the previous paragraph, this amounts to verifying the cofibrancy of  
$f_1 \square^{\circ_{\O}} \mathsf{F}(\beta \to \beta+1)$ for $\beta<\kappa$ (the condition for limit ordinals being automatic since $\circ_{\O}$ commutes with transfinite compositions in the second variable).
One hence reduces to the case where $f_2\colon A\to B$ is the pushout of a generating cofibration $\O\circ X \to \O \circ Y$, such as in Proposition \ref{FILT PROP}. 
Borrowing the notation from that proposition we see that it suffices to show that the vertical map of filtration $\omega$-diagrams (recall that $\omega$ denotes $(0 \to 1 \to 2 \to \cdots)$)
\[
\xymatrix@1{
A^M_{0} \ar[r]  \ar[d]  & A^M_{1} \ar[r]  \ar[d] & A^M_{2} \ar[r]  \ar[d] & A^M_3 \ar[r] \ar[d] & A^M_{4} \ar[d] \ar[r] & \quad \cdots \\
A^N_{0} \ar[r]  & A^N_{1} \ar[r] & A^N_{2} \ar[r] & A^N_3 \ar[r] & A^N_{4} \ar[r] & \quad \cdots}
\]
is a suitable $\omega$-projective cofibration. More explicitly, we need to show that each of the ``pushout corner maps'' 
$A^M_{r}\coprod_{A^M_{r-1}}A^{N}_{r-1}\to A^{N}_{r}$, $r \geq 0$ (note that $A^M_{-1}=A^N_{-1}=\**$) is a $S$ cofibration (resp.  monomorphism).
Using the inductive description (\ref{IND EQ}) this reduces to showing that the ``pushout corner maps'' of the diagrams
\[\xymatrix@1{
  M_A(r) \wedge_{\Sigma_r} Q^r_{r-1} \ar[r] \ar[d] &  M_A(r) \wedge_{\Sigma_r} Y^{\wedge r} \ar[d] \\
  N_A(r) \wedge_{\Sigma_r} Q^r_{r-1}        \ar[r] &  N_A(r) \wedge_{\Sigma_r} Y^{\wedge r}        }\]
are themselves $S$ cofibrations (resp. monomorphisms). 
Combining Theorems \ref{SIGMANPUSHPROD THM} and \ref{BIQUILLEN THM} this reduces to showing that $M_A(r)\to N_A(r)$, $r \geq 0$ is a $S$ cofibration (resp. monomorphism), or rather, that $M_A \to N_A$ is a $S$ cofibration (resp. monomorphism) in $\Sym$. Recalling from Definition \ref{MA DEF} that $M_A \to N_A$ can be written as 
\begin{equation}\label{MOA COF EQ} 
M \circ_{\O}(\O \amalg A) \to N \circ_{\O} (\O \amalg A),
\end{equation} 
we see that this last claim would follow directly from a different instance of the theorem we are trying to prove, namely the case of the maps $f_1\colon M \to N$ in $\mathsf{Mod}^r_{\O}$ and $\tilde{f_2} \colon \O \to \O \amalg A$ in $\mathsf{Mod}^l_{\O}$. Since $A$ is assumed cofibrant, it can be written as a retract of a transfinite composition of pushouts of generating cofibrations, and one hence reduces to the case $A = \colim_{\beta < \kappa} A_{\beta}$ where each $A_{\beta} \to A_{\beta+1}$ is the pushout of some generating positive $S$ cofibration $\O X_{\beta}\to \O Y_{\beta}$ in $\mathsf{Alg}_{\O}$.

Note now that one can repeat all of the arguments so far for $f_1$ and for the filtration 
$\tilde{f}_{2,\beta} \colon \O \amalg A_{\beta} \to \O \amalg A_{\beta+1}$
 of the map $\tilde{f_2} \colon \O \to \O \amalg A$. Firstly, repeating the ``$\kappa$-projective cofibration'' argument, the $\beta = 0$ condition is now that $f_1 \circ_{\O} \O=f_1$ is a $S$ cofibration (resp. monomorphism), which is just one of the hypotheses, and the limit ordinal condition is again automatic. One hence reduces to showing, by \textit{induction} on $\beta < \kappa$, that the theorem holds for $f_1$ and each $\tilde{f}_{2,\beta}$.  
Since $\tilde{f}_{2,\beta}$ is a pushout of $\O X_{\beta} \to \O Y_{\beta}$, one again reduces to showing that the map of filtration diagrams (built using Proposition \ref{FILT PROP} as described in Remark \ref{MODRED RMK})
\[\xymatrix@1{
 (\O \amalg A_{\beta})^M_{0} \ar[r] \ar[d] & (\O \amalg A_{\beta})^M_{1} \ar[r] \ar[d] & 
 (\O \amalg A_{\beta})^M_{2} \ar[r] \ar[d] & (\O \amalg A_{\beta})^M_{3} \ar[r] \ar[d] &  \quad \cdots \\
 (\O \amalg A_{\beta})^N_{0} \ar[r]        & (\O \amalg A_{\beta})^N_{1} \ar[r]        &
 (\O \amalg A_{\beta})^N_{2} \ar[r]        & (\O \amalg A_{\beta})^N_{3} \ar[r]        &  \quad \cdots      
}\]
is a suitable $\omega$-cofibration, and again one reduces to checking that the pushout corner maps of each diagram 
\begin{equation}\label{SECRED EQ}
\xymatrix@1{
  M_{\O \amalg A_{\beta}}(r) \check{\wedge}_{\Sigma_r} Q^r_{r-1,\beta} \ar[r] \ar[d] & 
  M_{\O \amalg A_{\beta}}(r) \check{\wedge}_{\Sigma_r} Y_{\beta}^{\check{\wedge}r}            \ar[d] \\
  N_{\O \amalg A_{\beta}}(r) \check{\wedge}_{\Sigma_r} Q^r_{r-1,\beta}        \ar[r] & 
	N_{\O \amalg A_{\beta}}(r) \check{\wedge}_{\Sigma_r} Y_{\beta}^{\check{\wedge}r}
}\end{equation}
are $S$ cofibrations (resp. monomorphisms) in $\Sym$. Arguing as before (but replacing uses of Theorems \ref{SIGMANPUSHPROD THM} and \ref{BIQUILLEN THM} by uses of their $\Sym$ analogues Propositions \ref{SIGMANPUSHPROD THM SYM} and \ref{BIQUILLEN THM SYM}) one reduces to checking that $M_{\O \amalg A_{\beta}} \to N_{\O \amalg A_{\beta}}$ is a $S$ cofibration (resp. monomorphism) in $\BSym$. The result now follows from the calculation in the proof of Proposition \ref{MOX PROP}, which identifies $M_{\O \amalg A_{\beta}}(r,s) \to N_{\O \amalg A_{\beta}}(r,s)$ with $M_{A_{\beta}}(r+s) \to N_{A_{\beta}}(r+s)$, together with the transfinite induction hypothesis (which, explicitly, states that $(M_{A_{\gamma}} \to N_{A_{\gamma}})_{\gamma \leq \beta}$ is a projective cofibration).

Tracing through the steps above we also see that indeed $f_1 \square^{\circ_{\O}} f_2$  will be a weak equivalence if either $f_1$ or the original $f_2$ is.

Finally, we explain what changes when $f_2$ is a general cofibration between cofibrant objects in $\mathsf{Mod}_{\O}^l$. Using Proposition \ref{MODRED PROP} to transfer the question to $\mathsf{Alg}_{\O^{\mathsf{s}}}(\Sym)$, all of the discussion above follows through by replacing uses of Theorems \ref{SIGMANPUSHPROD THM} and \ref{BIQUILLEN THM} by their $\Sym$ analogues, Propositions \ref{SIGMANPUSHPROD THM SYM} and \ref{BIQUILLEN THM SYM}. The only caveat is that when running the second filtration argument in the proof (specifically, when analyzing (\ref{SECRED EQ})), one instead uses the $\BSym$ analogues mentioned in Remark \ref{BSYM RMK}.
\end{proof}

\subsection{Proofs of Theorems \ref{MODMODELEXIST THM}, \ref{FORGETFUL THM}, \ref{BARCONSTRUCTION THM} and \ref{FIBERSEQ THM}}\label{PROOFS SEC}

We now derive Theorems \ref{MODMODELEXIST THM}, \ref{FORGETFUL THM}, \ref{BARCONSTRUCTION THM}, \ref{FIBERSEQ THM} from our main result, Theorem \ref{CIRCO POS THM}. Some of the proofs will make use of the following model structure on $\mathsf{Mod}^r_{\O}$.

\begin{theorem}\label{OINJMODEL THM}
Let $\O$ be an operad in $\mathsf{Sp}^{\Sigma}$.
There exists a cofibrantly generated model structure on $\mathsf{Mod}^r_{\O}$, which we call the {\bf monomorphism stable model structure}, such that cofibrations and weak equivalences are underlying in the monomorphism stable model structure on $\Sym$. Further, this is a left proper cellular simplicial model category.
\end{theorem}

\begin{proof}
This is a generalization of Theorem \ref{GSTABLE THM} and the same proof applies with only minor changes, hence we list only those.

Again one starts by proving a level equivalence result by verifying the conditions in \cite[Thm. 2.1.19]{Hov98}.
Choosing $\kappa$ to be an infinite cardinal larger than the number of simplices in $\O$ (counted over operadic, spectral and simplicial gradings), we define the set $I$ (resp. $J$) of generating cofibrations (resp. trivial cofibrations) to be a set of representatives of monomorphisms (resp. monomorphisms that are weak equivalences) between right modules with less than $\kappa$ simplices. Parts 1,2,3,4 of \cite[Thm. 2.1.19]{Hov98} are again immediate, and part 5 follows by noting that $I$ contains the maps $\left(S \otimes \left(\Sigma_m \times \Sigma_r \cdot \left( \partial\Delta^k_+\to \Delta^k_+ \right) \right)\right) \circ \O$. Part 6 reduces to showing a suitable ``$\kappa$ analogue'' of \cite[Lemma 5.1.7]{HSS}, and again the proof in \cite{HSS} generalizes by noting  that all relative homotopy groups have less than $\kappa$ elements and by building the $FC$ subspectra as sub-right modules rather than just subspectra. Left properness, cellularity and the simplicial model structure axioms are again straightforward.

To produce the desired stable version one again applies \cite[Thm. 4.1.1]{Hi03}, now localizing with respect to the set
 $\mathcal{S}_{\O}=\bigcup_{r \geq 0}(\Sigma_r \cdot \mathcal{S}_{\**})\circ \O$. 
That the resulting weak equivalences are as described follows by arguing exactly as in the last paragraph of the proof of Theorem \ref{SGSPEC THM}, using an identity adjunction to compare with the $\O$-projective model structure on $\mathsf{Mod}^r_{\O}$ over the $S$ stable model structure in $\Sym$ (this latter model structure is easily seen to exist by \cite[Lemma 2.3]{SS00}).
\end{proof}

\begin{proof}[Proof of Theorem \ref{MODMODELEXIST THM}]

To show the model structures exist it suffices (cf. \cite[Lemma 2.3]{SS00}) to check that for $J$ a set of generating trivial cofibrations, any transfinite composition of pushouts of maps in $\O \circ J$ is a weak equivalence. Noting that the proof of Theorem \ref{CIRCO POS THM} uses such a decomposition of $f_2$ and setting $f_1=\** \to \O$ it is always the case that $f_1 \circ_{\O} A$ is a monomorphism, so that repeating the first half of that proof one reduces to verifying that $\**=\**_A\to \O_A$ is a monomorphism, which is obviously the case even if $A$ is not cofibrant.

To verify the Quillen equivalence statement it suffices to show that the adjunction unit maps 
$ A\to \bar{\O}\circ_{\O} A$,
or
$(\O \to \bar{\O})\circ_{\O} A$, 
are weak equivalences whenever $A$ is cofibrant. Applying Theorem \ref{CIRCO POS THM} with $f_2=\O(0) \to A$ and noting that 
$f \circ_{\O} \O(0) = f \circ_{\O} \O \circ \** = f \circ \** = f(0)$ shows that the functor $(\minus)\circ_{\O} A$ preserves all weak equivalences that are also monomorphisms. It then follows from Theorem \ref{OINJMODEL THM} combined with Ken Brown's lemma \cite[Cor. 7.7.2]{Hi03} that $(\minus) \circ_{\O} A$ preserves all weak equivalences, finishing the proof.
\end{proof}

\begin{proof}[Proof of Theorem \ref{FORGETFUL THM}]
Apply Theorem \ref{CIRCO POS THM} to $f_1=(\** \to \O)$ and $f_2$ the intended cofibration between cofibrant objects.
\end{proof}

\begin{lemma}\label{CIRC POS LEMMA}
Consider positive $S$ cofibrations $f_i\colon A_i\to B_i$, $ 1\leq i\leq n$ in $\Sym$ all with positive $S$ cofibrant domains. 
Then their pushout product with respect to $\circ$,
$$\square^{\circ}(f_1,f_2,\cdots,f_n)$$
is a positive $S$ cofibration in $\Sym$ between positive $S$ cofibrant objects in $\Sym$, which is a weak equivalence if any of the $f_i$ is.
\end{lemma}

\begin{proof}
The proof follows by induction on $n$.

The case $n=2$ is a essentially a particular case of Theorem \ref{CIRCO POS THM} with $\O=\mathcal{I}$, except with an extra claim about positiveness. The extra claim follows by equation (\ref{CIRC DEF}) which shows that $(X\circ Y)_0(0)=\**$ if both $X_0(0)=\**$ and $Y_0(0)=\**$.

For the induction case, recalling that $\circ$ preserves colimits in the first variable yields  
$$\square^{\circ}(f_1,f_2,\cdots,f_n)=
\left(\square^{\circ}(f_1,f_2,\cdots,f_n-1)\right)\square^{\circ}f_n$$
(note however that the similar equation with brackets on the right fails), 
and the result follows by combining the induction hypothesis with the $n=2$ case.
\end{proof}

\begin{proof}[Proof of Theorem \ref{BARCONSTRUCTION THM}]
Recall that the degeneracies of $B_n(M,\O,N)$ are formed using only the unit map $\eta \colon \mathcal{I}\to \O$. The result now follows from Lemma \ref{CIRC POS LEMMA} since the maps whose cofibrancy must be verified are the maps
\[\square^{\circ}(\** \to M, \eta, \cdots, \eta,\** \to N),\]
where $\eta$ is allowed to appear any number of times.
\end{proof}

\begin{proof}[Proof of Theorem \ref{FIBERSEQ THM}]
By the existence of the monomorphism (resp. $\O$-pro\-jec\-tive $S$) stable model structure on $\mathsf{Mod}^r_{\O}$ (cf. Theorem \ref{OINJMODEL THM} (resp. its proof)) together with the fact that colimits (resp. limits) are underlying, homotopy cofiber (resp. fiber) sequences in $\mathsf{Mod}^r_{\O}$ match the underlying homotopy cofiber (resp. fiber) sequences in $\Sym$. Therefore, homotopy fiber and cofiber sequences in $\mathsf{Mod}^r_{\O}$ coincide since $\Sym$ is stable. Noting that the argument in the proof of Theorem \ref{MODMODELEXIST THM} shows $(\minus) \circ_{\O} A$ is already a left derived functor, and hence preserves homotopy cofiber sequences, finishes the proof.
\end{proof}

\end{document}